\documentclass[12pt,reqno]{amsart}
\usepackage{latexsym,amsmath,amssymb,amsfonts,amscd,graphics,appendix,amsxtra}
\usepackage{mathrsfs}
\usepackage{graphicx}
\usepackage[normalem]{ulem}
\usepackage{comment}
\usepackage{makecell}

\usepackage[utf8]{inputenc}
\usepackage{tikz}
\usetikzlibrary{matrix,arrows,decorations.pathmorphing}
\usepackage{tikz-cd}
\usepackage[all]{xy}
\usepackage{fullpage}

\usepackage{float}
\floatstyle{plaintop}
\restylefloat{table}
\usepackage{booktabs}

\usepackage{array}
\newcolumntype{P}[1]{>{\centering\arraybackslash}p{#1}}

\usepackage{hyperref}

\newtheorem{theorem}{Theorem}[section]
\newtheorem{lemma}[theorem]{Lemma}
\newtheorem{proposition}[theorem]{Proposition}
\newtheorem{corollary}[theorem]{Corollary}

\theoremstyle{definition}
\newtheorem{definition}[theorem]{Definition}
\newtheorem{example}[theorem]{Example}

\theoremstyle{remark}
\newtheorem{remark}[theorem]{Remark}

\numberwithin{equation}{section}

\hypersetup{
    colorlinks=true,
    linkcolor=red,
    filecolor=magenta,      
    urlcolor=red,
}

\tikzset{
    partial ellipse/.style args={#1:#2:#3}{
        insert path={+ (#1:#3) arc (#1:#2:#3)}
    }
}

\title{
Unimodal singularities and boundary divisors in the KSBA moduli of a class of Horikawa surfaces
}

\author{Patricio Gallardo}
\address{Department of Mathematics, University of California, Riverside, CA 92501, USA }
\email{pgallard@ucr.edu}
\urladdr{https://profiles.ucr.edu/app/home/profile/pgallard}

\author{Gregory Pearlstein}
\address{Department of Mathematics, University of Pisa, 56127 Italy.  On leave, 
Department of Mathematics, Texas A\&M University, College Station, TX 77843-3368}
\email{greg.pearlstein@unipi.it}

\author{Luca Schaffler}
\address{Dipartimento di Matematica e Fisica, Universit\`a degli Studi Roma Tre, Largo San Leonardo Murialdo 1, 00146, Roma, Italy}
\email{luca.schaffler@uniroma3.it}
\urladdr{https://ricerca.matfis.uniroma3.it/users/lschaffler/}

\author{Zheng Zhang}
\address{Institute of Mathematical Sciences, ShanghaiTech University, Shanghai 201210, China}
\email{zhangzheng@shanghaitech.edu.cn}
\urladdr{https://sites.google.com/site/zhengzhangmathhomepage/}

\begin{document}

\begin{abstract}
Smooth minimal surfaces of general type with $K^2=1$, $p_g=2$, and $q=0$ constitute a fundamental example in the geography of algebraic surfaces, and the 28-dimensional moduli space $\mathbf{M}$ of their canonical models admits a modular compactification $\overline{\mathbf{M}}$ via the minimal model program.  We describe eight new irreducible boundary divisors in such compactification parametrizing reducible stable surfaces. Additionally, we study the relation with the GIT compactification of $\mathbf{M}$ and the Hodge theory of the degenerate surfaces that the eight divisors parametrize.
\end{abstract}

\maketitle


\section{Introduction}

The interplay between geometric compactifications and Hodge theory is one of the driving forces in moduli theory. Well-studied cases include abelian varieties \cite{Ale02}, K3 surfaces \cite{AE21,AEH21,ABE22,AE22,AET23}, algebraic curves \cite{CV11}, cubic surfaces \cite{GKS21}, cubic fourfolds \cite{Laz10}, etc.  Recently, there has been a great focus on generalizing this interplay in the case of surfaces of general type --- the so-called non-classical cases \cite{KU09,GGLR17}. A particular well-posed case for such generalization is given by algebraic surfaces of general type with geometric genus $p_g=2$, irregularity $q=0$, and $K^2=1$. These were described by Enriques \cite{Enr49} and later studied by Horikawa \cite{Hor76}. In the current work we refer to these surfaces, which are sometimes called I-surfaces, simply as Horikawa surfaces. By the work of Gieseker \cite{Gie77}, the canonical models of such surfaces form a $28$-dimensional quasi-projective coarse moduli space $\mathbf{M}$.

In the current paper, we consider the compactification perspective offered by the minimal model program. By the work of Koll\'ar, Shepherd-Barron, and Alexeev \cite{KSB88,Ale96,Kol23} there exists a projective and modular compactification parametrizing stable surfaces with $K^2=1$ and $\chi(\mathcal{O})=3$, which contains $\mathbf{M}$. This is also known as the KSBA compactification, and in this specific case has several irreducible components by \cite{FPRR22}. In this work we focus on the main irreducible component $\overline{\mathbf{M}}$ of this compactification, which parametrizes smoothable surfaces. Our first contribution towards the understanding of $\overline{\mathbf{M}}$ is the following theorem that combines Theorems~\ref{thm:geometry-of-Y-and-Z} and \ref{thm:KSBA-strata-are-divisors}. 

\begin{theorem}
\label{thm:main1}
Let $\Sigma$ be any of the following eight non-log canonical isolated unimodular surfaces singularities: 
\[
E_{12},~E_{13},~E_{14},~Z_{11},~Z_{12},~Z_{13},~W_{12},~W_{13},
\]
which are classified by \cite{Arn76} and appear as the unique singularity of the degenerations of Horikawa surfaces described in Definition~\ref{def:degenerations-Horikawa-we-compute-stable-replacement-of}. Then, there exists an irreducible boundary divisor $\mathbf{D}_{\Sigma}\subseteq\overline{\mathbf{M}}$ generically parametrizing stable surfaces $S_{\Sigma}$ which are the gluing of two surfaces 
$\widetilde{Y}_{\Sigma}$ and $\widetilde{Z}_{\Sigma}$ along a $\mathbb{P}^1$.   The surfaces
$\widetilde{Y}_{\Sigma}$ and $\widetilde{Z}_{\Sigma}$ satisfy $h^2(\mathcal{O})=1$, $h^1(\mathcal{O})=0$, and have only finite cyclic quotient singularities. 
Moreover, $\widetilde{Y}_{\Sigma}$ is an ADE K3 hypersurface within a weighted projective space, see Table~\ref{table:intro}.
\end{theorem}

The stable limits of surfaces in Theorem~\ref{thm:main1}
are constructed in \S\,\ref{sec:stable-replacement-degen-Horikawa-surfaces} via
explicit stable replacement. 
In \S\,\ref{sec:dim-count-for-boundary-strata} we prove that these degenerate stable surfaces describe divisors in the moduli space. In fact, we prove a
bit more. Recall that a smooth Horikawa surface has $h^{1,1}=29$.  For each of the singularity types listed in Theorem \ref{thm:main1},
the subscript is equal to the Milnor number 
$\mu_{\Sigma}$ of the singularity 
(e.g. $E_{12}$ has $\mu_{E_{12}}=12$). The moduli count for the surfaces $\widetilde Y_{\Sigma}$ is $\mu_{\Sigma}-2$ whereas the moduli count for the surfaces $\widetilde Z_{\Sigma}$ is $29-\mu_{\Sigma}$.  
The fact that ${\bf D}_{\Sigma}$ is a divisor corresponds to the fact
that we can deform $\widetilde{Y}_{\Sigma}$ and $\widetilde{Z}_{\Sigma}$ 
independently.  In Corollaries~\ref{cor:top-Euler-char-Ytilde} and 
~\ref{cor:top-Eul-char-tildeZ} we show that   
$\chi_{\mathrm{top}}(\widetilde{Y}_{\Sigma}) = \mu_{\Sigma} + 3$ and 
$\chi_{\mathrm{top}}(\widetilde{Z}_{\Sigma})=36-\mu_\Sigma$.  Therefore, 
$\chi_{\mathrm{top}}(\widetilde{Y}_{\Sigma}\cup\widetilde{Z}_{\Sigma}) = (36-\mu_\Sigma) + (\mu_\Sigma+3) - \chi_{\mathrm{top}}(\mathbb P^1) = 37$.  

\begin{remark}
Our work is complementary to the work of Coughlan, Franciosi, Pardini, Rana, and Rollenske. In \cite{FPR17}, the authors described the locus $\mathbf{M}^{\mathrm{Gor}}\subseteq\overline{\mathbf{M}}$ parametrizing Gorenstein stable surfaces. The dimension of the boundary of $\mathbf{M}^{\mathrm{Gor}}$ is $20$ (not pure), and it parametrizes surfaces with at worst elliptic singularities of degree $1$ and $2$ (also known as $\widetilde{E}_8$ and $\widetilde{E}_7$ respectively). In \cite{FPRR22}, the authors found two boundary divisors $\mathbf{D}_1,\mathbf{D}_2\subseteq\overline{\mathbf{M}}$ generically parametrizing stable surfaces with precisely one isolated singularity of type $\frac{1}{4}(1,1)$ and $\frac{1}{18}(1,5)$ (see also \cite[Example~1.3.1]{Hac16}). In \cite{CFPRR23}, the authors identified a third boundary divisor $\mathbf{D}_3$ generically parametrizing stable surfaces with a $\frac{1}{25}(1,14)$ singularity and of cuspidal type. The divisor $\mathbf{D}_3$ is precisely the intersection of $\overline{\mathbf{M}}$ with another irreducible component of the moduli space of stable surfaces with $K^2=1$ and $\chi(\mathcal{O})=3$ (the stable surfaces generically parametrized by such component were first constructed in \cite{RU19}). We remark that the surfaces generically parametrized by $\mathbf{D}_1,\mathbf{D}_2,\mathbf{D}_3$ are irreducible, while the ones generically parametrized by $\mathbf{D}_\Sigma$ have two irreducible components.

The boundary divisors $\mathbf{D}_\Sigma$ are not necessarily the only strata in the KSBA compactification $\overline{\mathbf{M}}$ associated to the singularities $\Sigma$. As discussed above, a full description of this compactification is currently an effort lead by multiple groups, e.g. \cite{FPR17,CFPRR23,FPRR22,CFPR22}.
\end{remark}

For our second result, we observe that there is a GIT compactification $\overline{\mathbf{M}}^{\mathrm{git}}$ of the moduli space $\mathbf{M}$, see \cite{Wen21}. Therefore, it is natural to compare it with $\overline{\mathbf{M}}$.  
We show that $\overline{\mathbf{M}}^{\mathrm{git}}$ essentially forgets the information
contained in $\widetilde Y_{\Sigma}$. More precisely, in \S\,\ref{sec:ExtendKSBAGIT} we prove the following.

\begin{theorem}
\label{thm:main2}
The birational map $\overline{\mathbf{M}}\dashrightarrow\overline{\mathbf{M}}^{\mathrm{git}}$ given by the identity on the interior $\mathbf{M}$ of the compactifications, extends to a dense open subset of the boundary divisors $\mathbf{D}_{\Sigma}$. If $f$ denotes such extension, then the relative dimension of $f|_{\mathbf{D}_{\Sigma}}$ equals $\mu_{\Sigma}-2$.
\end{theorem}

\begin{table}
\caption{Geometric features of the stable surfaces $S_{\Sigma}$.}
\label{table:intro}
\begin{center}
\renewcommand{\arraystretch}{1.4}
\begin{tabular}{|c|c|c|c|c|c|c|c|c|c|}
\hline
Sing. $\Sigma$ & $E_{12}$ & $E_{13}$ & $E_{14}$ & $Z_{11}$ & $Z_{12}$ & $Z_{13}$ & $W_{12}$ & $W_{13}$ & \\
\hline
No. $\widetilde{Y}_\Sigma$ in \cite[\S\,13.3]{IF00}
& $88$ & $70$ & $53$ & $71$ & $51$ & $35$ & $41$ & $30$
&
see \S\,\ref{subsec:double-cover-of-Y}
\\ 
\hline
$K_{\widetilde{Z}_\Sigma}^2$ & $\frac{19}{21}$ & $\frac{4}{3}$ & $\frac{2}{3}$ & $\frac{5}{6}$ & $\frac{2}{3}$ & $\frac{7}{15}$ & $\frac{3}{5}$ & $\frac{1}{3}$ 
&
see Prop.~\ref{prop:double-cover-of-Z-invariants}
\\
\hline
\end{tabular}
\end{center}
\end{table}

\par Hodge theoretically, the degenerations generically parametrized by the boundary divisors $\mathbf{D}_{\Sigma}$ constructed in this paper are the analogs of curves of compact type. More precisely, the period map has finite monodromy about the generic point of each $\mathbf{D}_{\Sigma}$.  Accordingly, the 
period map extends holomorphically across one-dimensional arcs meeting $\mathbf{D}_{\Sigma}$ transversely at a generic point.  More generally, in \S\,\ref{sec:Hodge-theory-part} we prove the following results, which are applicable in our setting.

\begin{theorem}[Theorem~\ref{thm:LMHS-is-pure}]
\label{thm:hodge-intro-1} Let $\pi\colon\mathcal{S}\to\Delta$ be a one-parameter
degeneration of complex projective surfaces which is smooth over 
$\Delta^* =\Delta\setminus\{0\}$ such that
\begin{itemize}
    \item[(a)] If $t\neq 0$ then $S_t=\pi^{-1}(t)$ has geometric genus $2$.
    \item[(b)] The central fiber $S_0=\pi^{-1}(0)$ is the union of two irreducible components
    $\widetilde{Y}$ and $\widetilde{Z}$, each of which has $h^2(\mathcal{O})=1$ and at 
    worst rational singularities.
\end{itemize}
Then, the local system $\mathcal{V}_{\mathbb Q} = R^2\pi_*(\mathbb Q)$ over $\Delta^*$ has finite monodromy.
\end{theorem} 

\par To determine the limit mixed Hodge structure of Theorem~\ref{thm:hodge-intro-1}, we recall that given a semistable degeneration $\mathcal{X}\to\Delta$, the Clemens--Schmid sequence relates the mixed Hodge structure of the central fiber $X_0$ with the limit mixed Hodge structure of the period map, which is constructed on a generic fiber $X_{\eta}$.  We also recall that one of the basic Hodge theoretic birational invariants of a complex projective surface $X$ is its transcendental lattice $T(X)$.  Let $T[H^2(X,\mathbb Q)]$
denote the underlying rational Hodge structure of $T(X)$.

\begin{theorem}[Corollary~\ref{cor:transcendental-part-breaks-over-Q}]
Let $\mathcal{S}\to\Delta$ be as in Theorem~\ref{thm:hodge-intro-1}, and assume that $H^2(\widetilde{Y},\mathbb{Q})$ and $H^2(\widetilde{Z},\mathbb{Q})$ are pure of
weight $2$.  Let $\widehat{\mathcal{S}}\to\widetilde{\Delta}$ be a semistable degeneration obtained from $\mathcal{S}\to\Delta$ via a composition of covers and birational modifications of the central fiber, so that its central fiber $\widehat{S}_0$ is reduced and simple normal crossing. Also denote by $H^2_{\lim}(\widehat{S}_{\eta},\mathbb Q)$ the limit mixed Hodge structure on a generic fiber $\widehat{S}_{\eta}$ of $\widehat{\mathcal{S}}\to\widetilde{\Delta}$. Then,
$$
     T[H^2_{\lim}(\widehat{S}_{\eta},\mathbb Q)]
     \cong T[H^2(\widetilde Z,\mathbb Q)]
    \oplus T[H^2(\widetilde Y,\mathbb Q)]
$$
where $T[A]$ is the transcendental part of a $\mathbb Q$-Hodge structure $A$ of weight 2 with $F^3A=0$ (cf. Definition \eqref{def:transcendental-part}).
\end{theorem}

\begin{remark} 
In our setting (see Theorem~\ref{thm:main1}), $H^2(\widetilde{Z}_{\Sigma})$ carries a pure Hodge structure of weight 2 since $\widetilde{Z}_{\Sigma}$ is a V-manifold as it has only finite quotient singularities \cite{PS08}.  Likewise, the Hodge structure on $H^2(\widetilde{Y}_{\Sigma})$ is pure  because $\widetilde{Y}_{\Sigma}$ is an ADE K3 surface.
\end{remark}

In the case of singularities of type $Z_{11}$, $Z_{12}$,
$Z_{13}$, $W_{12}$, and $W_{13}$, the associated surface $\widetilde{Z}_{\Sigma}$ appearing in Theorem~\ref{thm:main1} is birational to a K3 surface which can be presented as the double cover of $\mathbb P^2$
branched along a sextic curve $C_{\Sigma}$.  The geometry of the 
curves $C_{\Sigma}$ is described in Proposition~\ref{prop:birational-equiv}.  The singular locus of a generic curve $C_{\Sigma}$ is $\emptyset$, $\{A_1\}$, $\{A_2\}$, $\emptyset$, $\{A_1\}$ if $\Sigma=Z_{11}$, $Z_{12}$, $Z_{13}$, $W_{12}$, $W_{13}$ respectively.


\subsection*{Acknowledgments}
We would like to thank Valery Alexeev, Patrick Brosnan, Sebastian Casalaina-Martin, Marco Franciosi, Matt Kerr, Radu Laza, Rita Pardini, Chris Peters, Julie Rana, and Filippo Viviani for helpful conversations. We also thank the anonymous referees for the valuable comments and suggestions. P. Gallardo thanks the University of California, Riverside and Washington University at St. Louis for the welcoming environment. G. Pearlstein was partially supported by the projects PRIN2022 ``Geometry of Algebraic Structures: Moduli, Invariants, Deformations'' and the University of Pisa PRA\verb|_|2018\verb|_|5 ``Spazi di moduli, rappresentazioni e strutture combinatorie''. L. Schaffler was partially supported by the projects ``Programma per Giovani Ricercatori Rita Levi Montalcini'', PRIN2017 ``Advances in Moduli Theory and Birational Classification'', PRIN2020 ``Curves, Ricci flat varieties and their Interactions'', and, while at KTH, by a KTH grant by the Verg foundation. G. Pearlstein and L. Schaffler are members of the INdAM group GNSAGA. Z. Zhang is supported in part by NSFC grant 12201406. The authors would also like to thank the Pacific Institute for the Mathematical Sciences and Texas A\&M University for facilitating the conferences
Hodge Theory, Arithmetic and Moduli I \& II where this project took form. These conferences and related activities were funded by the host institutions as well as NSF grants DMS 1904692, 1361147, and 1361120.


\tableofcontents


\section{Preliminaries}


\subsection{Horikawa surfaces and their moduli space}
\label{sec:Horikawasurf}

Recall from the introduction that we call Horikawa surface a minimal smooth projective surface $S$ of general type satisfying
\[
h^2(\mathcal{O}_S)=2,~h^1(\mathcal{O}_S)=0,~\textrm{and}~K_S^2=1.
\]
In the literature, these specific surfaces are also called I-surfaces. For such surface $S$, the divisor $2K_S$ induces a degree $2$ morphism $S\rightarrow\mathbb{P}(1,1,2)$ with branch curve given by the vanishing of a weighted degree $10$ polynomial $F_{10}(x,y,z)$ \cite[\S\,2]{Hor76}. Therefore, if $[x:y:z:w]$ are the coordinates in $\mathbb{P}(1,1,2,5)$, then $S$ is isomorphic to the weighted degree $10$ hypersurface in $\mathbb{P}(1,1,2,5)$ given by
$w^2=F_{10}(x,y,z).$

Following the approach in \cite{Wen21}, we review the construction of the moduli space of Horikawa surfaces and its GIT compactification.

\begin{definition}
Let $F_{10}(x,y,z)$ be a weighted degree $10$ polynomial such that the coefficient of $z^5$ is nonzero. Up to rescaling the coefficients of $F_{10}(x,y,z)$ by an element of $\mathbb{C}^*$, we can assume that such coefficient is $1$. We can then expand $F_{10}(x,y,z)$ as follows:
\[
F_{10}(x,y,z)=z^5+q_2(x,y)z^4+q_4(x,y)z^3+q_6(x,y)z^2+q_8(x,y)z+q_{10}(x,y),
\]
where $q_d(x,y)$ are homogeneous polynomials of degree $d\in\{2,4,6,8,10\}$. After applying the following invertible change of coordinates (also known as a Tschirnhaus transformation):
\[
x=\tilde{x},~y=\tilde{y},~z=\tilde{z}-\frac{1}{5}q_2(\tilde{x},\tilde{y}),
\]
we obtain
\[
G_{10}(\tilde{x},\tilde{y},\tilde{z})=\tilde{z}^5+g_4(\tilde{x},\tilde{y})\tilde{z}^3+g_6(\tilde{x},\tilde{y})\tilde{z}^2+g_8(\tilde{x},\tilde{y})\tilde{z}+g_{10}(\tilde{x},\tilde{y}),
\]
for some new homogeneous polynomials $g_d(\tilde{x},\tilde{y})$ of degree $d\in\{4,6,8,10\}$. We call the above degree $10$ polynomial and the corresponding hypersurface $V(w^2-G_{10}(\tilde{x},\tilde{y},\tilde{z}))\subseteq\mathbb{P}(1,1,2,5)$ in \emph{normal form}.
\end{definition}

\begin{definition}
\label{moduli-and-family}
The vector space of coefficients for a polynomial in normal form
\[
F_{10}(x,y,z)=z^5+q_4(x,y)z^3+q_6(x,y)z^2+q_8(x,y)z+q_{10}(x,y)
\]
is the vector space
\[
\mathbb{V}_{10}=\bigoplus_{k=2}^5H^0(\mathbb{P}^1,\mathcal{O}_{\mathbb{P}^1}(2k))\cong\mathbb{C}^{32}.
\]
Let $\mathbf{U}$ be the open subset of points in the affine space $\mathrm{Spec}(\mathrm{Sym}(\mathbb{V}_{10}^\vee))\cong\mathbb{A}^{32}$ parametrizing smooth Horikawa surfaces in normal form. 
There is a natural $\mathrm{GL}_2$-action on $\mathbf{U}$ given by linear change of coordinate in $x$ and $y$. By \cite[Lemma~4 and Lemma~7]{Wen21}, we have that the points of the $28$-dimensional quotient $\mathbf{M}=\mathbf{U}/\mathrm{GL}_2$ are in bijection with the isomorphism classes of Horikawa surfaces considered. 
We refer to $\mathbf{M}$ as the coarse \emph{moduli space of Horikawa surfaces}. 
This coincides with the Gieseker moduli space of canonical surfaces $S$ of general type with $K_S^2=1,p_g(S)=2$, and $q(S)=0$ \cite{Gie77}. The family of Horikawa surfaces is defined over $\mathbf{U}$ by the relative equation
\[
\mathcal{H}'=V(w^2-z^5-g_4(x,y)z^3-g_6(x,y)z^2-g_8(x,y)z-g_{10}(x,y))\subseteq\mathbf{U}\times\mathbb{P}(1,1,2,5),
\]
with proper flat morphism $\mathcal{H}'\rightarrow\mathbf{U}$ given by the restriction of the projection onto the first factor. This family descends to the geometric quotient by $\mathrm{GL}_2$ giving the family of Horikawa surfaces $\mathcal{H}\rightarrow\mathbf{M}$ with pairwise non-isomorphic fibers.
\end{definition}

\begin{definition}
\label{def:GIT}
As discussed in \cite[Proposition~9]{Wen21}, we have a natural projective GIT compactification of $\mathbf{M}$ given by
\[
\overline{\mathbf{M}}^{\mathrm{git}}:=\mathbb{A}^{32}/\!\!/\mathrm{GL}_2\cong(\mathbb{A}^{32}/\mathbb{C}^*)/\!\!/(\mathrm{GL}_2/\mathbb{C}^*)\cong\mathbb{P}(4^5,6^7,8^9,10^{11})/\!\!/\mathrm{SL}_2,
\]
where the GIT quotients are with respect to appropriate linearizations (see \cite{Wen21}) and $\mathbb{P}(4^5,6^7,8^9,10^{11})$ denotes a weighted projective space with $n^m$ representing $n,\ldots,n$ repeated $m$-times.
\end{definition}


\subsection{Stable pair compactification}

By the work of Koll\'ar, Shepherd-Barron, and Alexeev \cite{Ale96,KSB88} we can construct a projective compactification $\mathbf{M}\subseteq\overline{\mathbf{M}}$ which is a coarse moduli space parametrizing \emph{stable} surfaces. Let us review the main definitions of interest.

\begin{definition}
Let $X$ be a variety and $\sum_ib_iB_i$ a $\mathbb{Q}$-divisor on $X$ with $0<b_i\leq1$ and $B_i$ prime divisors. Then the pair $(X,B)$ is called \emph{stable} provided it has \emph{semi-log canonical singularities} \cite[Definition--Lemma~5.10]{Kol13} and $K_X+B$ is ample. If $B=0$, then $X$ is called a \emph{stable variety}.
\end{definition}

Stable pairs with divisor $B=0$ can be used to construct a geometric, functorial, and projective compactification of $\mathbf{M}$. Here we follow \cite[Definitions~1.4.2, 1.4.3]{Ale15}.

\begin{definition}
\label{def:Viehweg-moduli-stack}
Let $d,N$ be positive integers and $C$ a positive rational number. For any reduced complex scheme $S$, define $\overline{\mathcal{V}}(S)=\overline{\mathcal{V}}_{d,N,C}(S)$ to be the set of proper flat families $\mathcal{X}\rightarrow S$ with the following properties:
\begin{itemize}

\item Every geometric fiber $X_s$ is a stable variety, $\dim(X_s)=d$, and $K_{X_s}^d=C$;

\item There exists an invertible sheaf $\mathcal{L}$ on $\mathcal{X}$ such that for every geometric fiber $X_s$, $\mathcal{L}|_{X_s}\cong\mathcal{O}_{X_s}(N(K_{X_s}))$.

\end{itemize}
The above stack $\overline{\mathcal{V}}$ is called the \emph{Viehweg moduli stack}. In the notation of \cite[\S\,8]{Kol23}, this moduli functor is $\mathcal{SP}(0,d,C)$.
\end{definition}

\begin{definition}
\label{def:KSBA-comp-mod-Horikawa-surf}
Consider the Viehweg moduli stack $\overline{\mathcal{V}}$ for $d=2$, $C=1$, and $N$ large enough (see Definition~\ref{def:the-eight-boundary-divisors}). Let $\overline{\mathbf{V}}$ be the corresponding projective coarse moduli space. The family of Horikawa surfaces $\mathcal{H}\rightarrow\mathbf{M}$ in Definition~\ref{moduli-and-family} induces a morphism $h\colon\mathbf{M}\rightarrow\overline{\mathbf{V}}$ which is injective on $\mathbb{C}$-points. We denote by $\overline{\mathbf{M}}$ the normalization of the closure of the image of $h$ in $\overline{\mathbf{V}}$, and we will refer to $\overline{\mathbf{M}}$ as the \emph{KSBA compactification of the moduli space of Horikawa surfaces}.
\end{definition}

\begin{remark}
Equivalently, we could have defined $\overline{\mathbf{M}}$ using the Koll\'ar moduli stack \cite[Definition~1.4.2]{Ale15}. We chose to work with $\overline{\mathcal{V}}$ instead because it simplifies the discussion when constructing the boundary divisors $\mathbf{D}_\Sigma$ in Definition~\ref{def:the-eight-boundary-divisors}.
\end{remark}


\subsection{Hypersurface singularities}
\label{Hypersurf-sings}
The GIT compactification $\overline{\mathbf{M}}^{\mathrm{git}}$ given in Definition~\ref{def:GIT} parametrizes classes of degenerations of degree $10$ hypersurfaces in $\mathbb{P}(1,1,2,5)$ that appear as double covers of $\mathbb{P}(1,1,2)$. Therefore, we can determine the surface singularities by considering the singularities of the branch curve 
$V(F_{10}(x,y,z)) \subseteq \mathbb{P}(1,1,2)$. If the curve is away from the singularity of $\mathbb{P}(1,1,2)$, then the singularities are locally isomorphic to a plane curve singularity. We now describe the plane curve singularities of interest for the current work.

Given a singularity $T$, we can associate two invariants: the Milnor number $\mu(T)$ and the modality $m(T)$ of the singularity (also called modulus). These invariants measure the complexity of the singularity and provide a criterion to classify them. From the moduli theory perspective, the modality $m(T)$ is particularly interesting because it is equal to the dimension of stratum in the base space of the versal deformation where $\mu$ is constant, minus $1$ (see \cite[\S\,4]{Arn76}). Another relevant fact is that both $\mu(T)$ and $m(T)$ are upper semicontinuous, see \cite[Chapter~I, \S\,2]{GLS07}.
The hypersurface singularities with $m(T)=0$ are precisely the ADE singularities. Therefore, the next case of interest are the singularities with $m(T)=1$. By the classification of log canonical two dimensional hypersurface singularities \cite[Table~1]{LR12} in combination with \cite[\S\,I.1]{Arn76}, we obtain that there are eight plane curve singularities of modality $1$ that are not log canonical. These are given in Table~\ref{tbl:eight-sing-types-local-models}. Here, the parameter $a$ is generic, and as it varies describes non isomorphic plane curves with the given isolated singularities.

\begin{table}[ht!]
    \centering
    \caption{Local models of the eight isolated non-log canonical singularities of modality $1$ that can be attained as degeneration in $\mathbb{P}(1,1,2,5)$ of Horikawa surfaces.}
\label{tbl:eight-sing-types-local-models}
    \begin{tabular}{c|c||c|c||c|c}
$E_{12}$ & $z^3+y^7+ay^5z$   
 & $Z_{11}$ & $yz^3 + y^5 +ay^4z$
& $W_{12}$ & $z^4+y^5+ay^3z^2$
\\
$E_{13}$ & $z^3+y^5z +ay^8$ 
&  $Z_{12}$ & $yz^3 + y^4z + ay^3z^2$ 
 & $W_{13}$ & $z^4 + y^4z + ay^6$
\\
$E_{14}$ & $z^3+ y^8 +ay^6z$  & $Z_{13}$ & $yz^3 + y^6 + ay^5z$
    \end{tabular}
\end{table}
We remark that we label the variables in a different way than in \cite[\S\,I.1]{Arn76}.  Our choice is to be compatible with the discussion that follows. Furthermore, the above notation denotes germs of singularities up to stable equivalence, that is, up to terms of the form $x_i^2$. Therefore, for instance, we use the notation $E_{12}$ for both the plane curve singularity and for the associated surface singularity 
$w^2=z^3 +y^7+ay^5z$. Summarizing, we have the following lemma.

\begin{lemma}
\label{lem:9-exc-fam-we-consider}
Let $\mathcal{X} \to \Delta$ be a one-dimensional family of smooth Horikawa surface degenerating a double cover of 
$\mathbb{P}(1,1,2)$ with an unique isolated singularity of modality $1$. Then the singularity is one of the following:
\[
E_{12},~E_{13},~E_{14},~Z_{11},~Z_{12},~Z_{13},~W_{12},~W_{13}.
\]
\end{lemma}

We conclude with the following diagram in Figure~\ref{fig:adjacency},
where the arrow $A \leftarrow B$ means that the germ of an $A$ singularity degenerates to the germ of a $B$ singularity.
\begin{figure}[ht!]
\begin{tikzcd}
\widetilde{E}_8 & \arrow{l} E_{12} & \arrow{l} E_{13} & \arrow{l} E_{14} & \widetilde{E}_7 & \arrow{l} Z_{11} & \arrow{l} Z_{12} & \arrow{l} Z_{13}
\\
& & & & & \arrow{u} W_{12} & \arrow{l} \arrow{u} W_{13} 
\end{tikzcd}
\caption{Adjacency diagram for the non-log canonical singularities within Theorem~\ref{thm:main1}, notation as \cite[\S\,I.1]{Arn76}. $\widetilde{E}_7$ and $\widetilde{E}_8$ are the simple elliptic singularities of degree $2$ and $1$, see
Remark~\ref{rmk:CareAdjacency}.}
\label{fig:adjacency}
\end{figure}
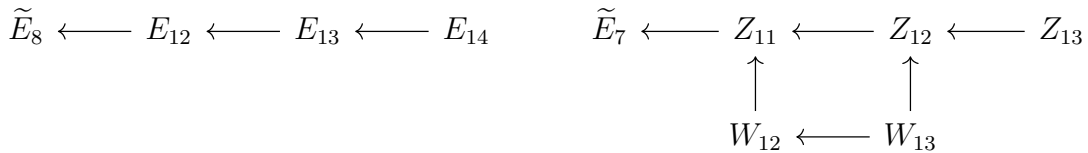
\begin{remark}
\label{rmk:CareAdjacency}
Within the context of the main theorems in the introduction (e.g. Theorem~\ref{thm:main1}), we highlight that 
the adjacency diagrams in Figure~\ref{fig:adjacency} do not mean that the divisors $\mathbf{D}_{\Sigma}$ are contained in each other or that they intersect generically in $\overline{\mathbf{M}}$. Additionally, in terms of deformation theory, we note that the germs of $\widetilde{E}_7$ and $\widetilde{E}_8$ degenerate to $Z_{11},Z_{12},Z_{13}$ and $E_{12},E_{13},E_{14}$ respectively. The equations of $\widetilde{E}_7$ and $\widetilde{E}_8$ are 
$x^4+y^4 + ax^2y^2=0$  with $a^2 \neq 4$ and 
$x^3 + y^6 + ax^2y^2=0$ with $4a^3 + 27 \neq 0$ respectively. 
However, after stable replacement, these describe codimension one strata of the boundary of $\overline{\mathbf{M}}$. In general, in a compact moduli space of stable varieties, it is an open problem to understand the reciprocal relations among the boundary strata corresponding to the singularities in an adjacency diagram of germ of singularities.
\end{remark}

\subsection{Weighted blow ups}

A key step in our work is a partial resolution of isolated singularities via weighted blow ups at a point, which, we briefly describe here. See \cite[\S\,6.38]{KSC04} for a reference. Let $(a_1,\ldots,a_n)$ be a sequence of relatively prime positive integers. We have a natural map 
\begin{align*}
\mathbb{A}^n&\dashrightarrow \mathbb{P}(a_1,\ldots,a_n)\\
(x_1,\ldots,x_n)&\mapsto(x_1^{a_1},\ldots,x_n^{a_n}).
\end{align*}
The weighted blow up of $\mathbb{A}^n$ with local coordinates $(x_1,\ldots,x_n)$ and weights $(a_1,\ldots,a_n)$ 
at the origin is the closure of the graph of the above rational map. Its exceptional divisor is isomorphic to $\mathbb{P}(a_1,\ldots,a_n)$, and the associated ideal is the integral closure of the ideal $(x_1^{N/a_1},\ldots,x_n^{N/a_n})$ for sufficiently divisible $N$. In particular, if $a_1=\ldots=a_n=1$, then we recover the simple blow up of $\mathbb{A}^n$ at the origin.


\section{Stable replacements of generic one-parameter degenerations of Horikawa surfaces}
\label{sec:stable-replacement-degen-Horikawa-surfaces}

In this section, we construct the \emph{KSBA stable replacements} for one-parameter degenerations of Horikawa surfaces over DVRs whose central fiber has a singularity of type $\Sigma$.  
We begin by first analyzing a special class of degenerations defined by a $\mathbb{C}^*$-action.  The general case of a DVR is handled in \S\,\ref{KSBA-DVR}.

\par Let $\mathbb C[x,y,z,w]$ denote the homogeneous coordinate ring of $\mathbb{P}(1,1,2,5)$ where $x$ and $y$ have degree $1$, $z$ has degree $2$, and 
$w$ has degree $5$.  Let $\Sigma$ be one of the eight exceptional families of isolated unimodular surface singularities in Lemma~\ref{lem:9-exc-fam-we-consider}. Let
\begin{equation*}
    \mathcal{S}(f)=\{([x:y:z:w],t)\in\mathbb{P}(1,1,2,5)\times\mathbb{C} \mid 
w^2-f(x,y,z)=0\},
\end{equation*}
where $f(x,y,z)$ is a homogeneous polynomial of weighted degree $10$ with coefficients in $\mathbb{C}[t]$.  Let $\pi\colon\mathcal{S}\to\mathbb{C}$ denote the projection onto the 
second factor with fibers $S_t = \pi^{-1}(t)$ and $\Delta$ be a neighborhood of $0\in\mathbb{C}$ such  that, after
restriction to $\Delta$, $\mathcal{S}$ becomes a one-parameter family of smooth Horikawa surfaces for $t\neq0$ and for $t=0$ the fiber $S_0$ has exactly one isolated singularity of type $\Sigma$. After a sequence of birational modifications of the central fiber of $\pi\colon\mathcal{S}\rightarrow\Delta$ and possibly base changes, we obtain a new family $\mathcal{S}'\rightarrow\Delta'$ whose central fiber $S_0'$ now has semi-log canonical singularities and it has ample canonical class, i.e. $S_0'$ is a stable surface. The surface $S_0'$ is called the stable replacement of the central fiber of $\mathcal{S}\rightarrow\Delta$. The isomorphism class of $S_0'$ corresponds to the limit point in the KSBA compactification $\overline{\mathbf{M}}$ of the arc $\Delta^\circ:=\Delta\setminus\{0\}\rightarrow\mathbf{M}$ induced by the family $\mathcal{S}^\circ:=\mathcal{S}\setminus S_0\rightarrow\Delta^\circ$. More precisely, for each singularity $\Sigma$ we describe some explicit families $\mathcal{S}\rightarrow\Delta$ such that the corresponding isomorphism classes of stable surfaces $S_0'$ generically describe a divisor in $\overline{\mathbf{M}}$. The discussion is organized as follows. First, in \S\,\ref{subsec:def-of-families} we define such families. Then, in \S\,\ref{subsec:computation-stable-replacement} we describe the stable replacements $S_0'$. Afterward, the remaining subsections contain the proofs of the claims in \S\,\ref{subsec:computation-stable-replacement}. We show that all the isomorphism classes of $S_0'$ give rise to boundary divisors in $\overline{\mathbf{M}}$ later in \S\,\ref{sec:dim-count-for-boundary-strata}.


\subsection{Definition of the families}
\label{subsec:def-of-families}

Let $\Sigma$ be one of the eight singularity types and let $(p,q),d$ as in Table~\ref{tbl:weights-and-degrees-for-eight-singularities}.
\begin{table}
\centering
\caption{Weights and degree associated to the eight singularities.}
\label{tbl:weights-and-degrees-for-eight-singularities}
\renewcommand{\arraystretch}{1.4}
\begin{tabular}{|c|c|c|c|c|c|c|c|c|}
\hline
Sing. & $E_{12}$ & $E_{13}$ & $E_{14}$ & $Z_{11}$ & $Z_{12}$ & $Z_{13}$ & $W_{12}$ & $W_{13}$ \\
\hline
$(p,q)$ & $(3,7)$ & $(2,5)$ & $(3,8)$ & $(3,4)$ & $(2,3)$ & $(3,5)$ & $(4,5)$ & $(3,4)$ \\
\hline
$d$ & $21$ & $15$ & $24$ & $15$ & $11$ & $18$ & $20$ & $16$ \\
\hline
\end{tabular}
\end{table}
The meaning of these constants is the following: if we assign weight $(p,q)$ to $(y,z)$, the lower degree part of the local singularity for $\Sigma$ becomes homogeneous of degree $d$, and this information will be used to compute the stable replacement of the central fiber $S_0$.

\begin{definition}
Let $\mathbb{V}_{10}$ be the complex vector space spanned by the monomials $x^ay^bz^c$ satisfying $a+b+2c=10$. If $x^ay^bz^c$ is a monomial in $\mathbb{V}_{10}$, then we define its weight with respect to $\Sigma$ as $\mathrm{wt}_\Sigma(x^ay^bz^c)=pb+qc-d$.
\end{definition}

It will be useful to keep track of which monomials have positive, negative, or zero weight across the eight singularity types.

\begin{proposition}
\label{prop:classification-signs-of-deg-10-monomials}
Consider the degree $10$ monomials in $\mathbb{P}(1,1,2)$. Identify the monomial $x^ay^bz^c$ with the triple $(a,b,c)$. Then the following $15$ monomials have positive weight across the eight singularity types:
\begin{gather*}
(0,0,5),~(1,1,4),~(2,2,3),~(0,2,4),~(3,3,2),~(1,3,3),~(2,4,2),~(0,4,3),\\
(1,5,2),~(2,6,1),~(0,6,2),~(1,7,1),~(0,8,1),~(1,9,0),~(0,10,0).
\end{gather*}
The following other $12$ monomials always have negative weight across the eight singularity types:
\begin{gather*}
(10,0,0),~(8,0,1),~(6,0,2),~(9,1,0),~(7,1,1),~(5,1,2),\\
(8,2,0),~(6,2,1),~(4,2,2),~(7,3,0),~(5,3,1),~(6,4,0).
\end{gather*}
The remaining nine monomials can be of positive, negative, or zero weight as the singularity type changes:
\begin{center}
\renewcommand{\arraystretch}{1.4}
\begin{tabular}{|c|c|c|c|c|c|c|c|c|}
\hline
$(a,b,c)$ & $E_{12}$ & $E_{13}$ & $E_{14}$ & $Z_{11}$ & $Z_{12}$ & $Z_{13}$ & $W_{12}$ & $W_{13}$ \\
\hline
$(4,0,3)$ & $0$ & $0$ & $0$ & $-$ & $-$ & $-$ & $-$ & $-$ \\
\hline
$(2,0,4)$ & $+$ & $+$ & $+$ & $+$ & $+$ & $+$ & $0$ & $0$ \\
\hline
$(3,1,3)$ & $+$ & $+$ & $+$ & $0$ & $0$ & $0$ & $-$ & $-$ \\
\hline
$(4,4,1)$ & $-$ & $-$ & $-$ & $+$ & $0$ & $-$ & $+$ & $0$ \\
\hline
$(5,5,0)$ & $-$ & $-$ & $-$ & $0$ & $-$ & $-$ & $0$ & $-$ \\
\hline
$(3,5,1)$ & $+$ & $0$ & $-$ & $+$ & $+$ & $+$ & $+$ & $+$ \\
\hline
$(4,6,0)$ & $-$ & $-$ & $-$ & $+$ & $+$ & $0$ & $+$ & $+$ \\
\hline
$(3,7,0)$ & $0$ & $-$ & $-$ & $+$ & $+$ & $+$ & $+$ & $+$ \\
\hline
$(2,8,0)$ & $+$ & $+$ & $0$ & $+$ & $+$ & $+$ & $+$ & $+$ \\
\hline
\end{tabular}
\end{center}
In particular, for each singularity type, there exist precisely two monomials of weight $0$ as listed below.
\begin{center}
\renewcommand{\arraystretch}{1.4}
\begin{tabular}{|c|c|c|c|c|c|c|c|c|}
\hline
Sing. & $E_{12}$ & $E_{13}$ & $E_{14}$ & $Z_{11}$ & $Z_{12}$ & $Z_{13}$ & $W_{12}$ & $W_{13}$ \\
\hline
$m_1(x,y,z)$ & $x^3y^7$ & $x^3y^5z$ & $x^2y^8$ & $x^5y^5$ & $x^4y^4z$ & $x^4y^6$ & $x^5y^5$ & $x^4y^4z$ \\
$m_2(x,y,z)$ & $x^4z^3$ & $x^4z^3$ & $x^4z^3$ & $x^3yz^3$ & $x^3yz^3$ & $x^3yz^3$ & $x^2z^4$ & $x^2z^4$ \\
\hline
\end{tabular}
\end{center}
\end{proposition}

\begin{proof}
This is an immediate computer assisted check.
\end{proof}

\begin{definition}
\label{def:weighted-subspaces-and-projections}
We let $U_\Sigma$ denote the codimension one subspaces of $\mathbb{V}_{10}$ consisting of elements for which $m_1$ and $m_2$ have the same coefficient.  Then, the weight function gives a direct sum decomposition $U_\Sigma=U_{\Sigma,+}\oplus U_{\Sigma,0}\oplus U_{\Sigma,-}$ where:
\begin{align*}
U_{\Sigma,+}&=\mathrm{Span}_\mathbb{C}\{x^ay^bz^c\in \mathbb{V}_{10}\mid\mathrm{wt}_\Sigma(x^ay^bz^c)>0\},\\
U_{\Sigma,0}&=\mathrm{Span}_\mathbb{C}\{m_1+m_2\},\\
U_{\Sigma,-}&=\mathrm{Span}_\mathbb{C}\{x^ay^bz^c\in \mathbb{V}_{10}\mid\mathrm{wt}_\Sigma(x^ay^bz^c)<0\}.
\end{align*}
We let $\pi_+$, $\pi_0$, and $\pi_-$ denote the projections from $U_\Sigma$ to $U_{\Sigma,+}$, $U_{\Sigma,0}$, and $U_{\Sigma,-}$ respectively.  If $W$ is a subspace of $U$, then we define $W_{\mathrm{reg}} = W \setminus \pi^{-1}_0(0)$ and $\mathbb P(W)_{\mathrm{reg}} = \{ [w]\in \mathbb P(W) \mid w\in W_{\mathrm{reg}}\}$. We have that $\mathbb{P}(W)_{\mathrm{reg}}$ is an affine patch of the projective space $\mathbb{P}(W)$.
\end{definition}

\begin{lemma}
\label{lem:decomposition-reg-elements}
$\mathbb P(U_\Sigma)_{\mathrm{reg}} \cong \mathbb P(U_{\Sigma,+}\oplus U_{\Sigma,0})_{\mathrm{reg}}\times \mathbb P(U_{\Sigma,0}\oplus U_{\Sigma,-})_{\mathrm{reg}}$.
\end{lemma}

\begin{proof}
Consider the following morphisms:
\begin{align*}
\varphi\colon\mathbb P(U_\Sigma)_{\mathrm{reg}}&\to \mathbb P(U_{\Sigma,+}\oplus U_{\Sigma,0})_{\mathrm{reg}}\times \mathbb P(U_{\Sigma,0}\oplus U_{\Sigma,-})_{\mathrm{reg}}\\
[u]&\mapsto([\pi_+(u)+ \pi_0(u)],[\pi_0(u)+\pi_-(u)]),
\end{align*}
\begin{align*}
\psi\colon \mathbb P(U_{\Sigma,+}\oplus U_{\Sigma,0})_{\mathrm{reg}}\times \mathbb P(U_{\Sigma,0}\oplus U_{\Sigma,-})_{\mathrm{reg}}&\to\mathbb P(U_\Sigma)_{\mathrm{reg}}\\
([\alpha],[\beta])&\mapsto[\alpha' + \beta'-m_1-m_2],
\end{align*}
where $\alpha'$, $\beta'$ are lifts of $[\alpha],[\beta]$ to $U_{\Sigma,+}\oplus U_{\Sigma,0}$ and $U_{\Sigma,0}\oplus U_{\Sigma,-}$ respectively such that $\pi_0(\alpha')=\pi_0(\beta') = m_1+m_2$. It can be checked directly that $\varphi$ and $\psi$ are inverse of each other.
\end{proof}

\begin{definition}
\label{def:t-star-action}
We now define a $\mathbb{C}^*$-action on $\mathbb{V}_{10}$ by describing how an element $t\in\mathbb{C}^*$ acts on a given monomial $x^ay^bz^c\in V$ of weight $\omega$. We define
$$
      t\star x^ay^bz^c =\left\{\aligned t^{-\omega}x^ay^bz^c\qquad &\textrm{if}~ \omega \leq 0 \\ x^a y^b z^c\qquad &\textrm{if}~ \omega >0.\endaligned\right.
$$
In particular $t\star (m_1 +m_2) = m_1 + m_2$, and hence this $\mathbb C^*$-action
descends to an action on $U_{\Sigma}$.
\end{definition}

\begin{remark}
\label{result-of-limit-to-zero}
If $v\in \mathbb{V}_{10}$ then $t\star v$ converges to $\pi_+(v)+\pi_0(v)$ as $t\to 0$, and
hence $t\star v\in\mathbb C[t]\otimes\mathbb{V}_{10}$.   Moreover, if 
$u\in(U_{\Sigma})_{\mathrm{reg}}$ then $\lim_{t\to 0}\, t\star u\neq 0$.
\end{remark}

\begin{definition}
\label{def:degenerations-Horikawa-we-compute-stable-replacement-of}
If $u\in(U_{\Sigma})_{\mathrm{reg}}$ we define the associated family of surfaces to be
\[
\mathcal{S}=\mathcal{S}(t\star u)=\{([x:y:z:w],t)\in\mathbb P(1,1,2,5)\times \mathbb{C} \mid w^2-t\star u=0\}\subseteq
\mathbb{P}(1,1,2,5)\times\mathbb{C}.
\]
The corresponding family of curves is
\begin{equation}
\label{eq:branch-curve-family}
\mathcal{B} = \mathcal{B}(t\star u) = \{([x:y:z],t)\in\mathbb P(1,1,2)\times \mathbb{C} \mid t\star u=0\}\subseteq
\mathbb{P}(1,1,2)\times\mathbb{C}
\end{equation}
with projection $\pi\colon\mathcal{B}\to\mathbb{C}$ and fiber 
$B_t =\pi^{-1}(t)\subseteq\mathbb P(1,1,2)$.
\end{definition}

\begin{remark}
The family $\mathcal{B}$ depends only on $[u]\in\mathbb P(U_\Sigma)_{\mathrm{reg}}$. The limit branch curve
\begin{equation}
\label{eq:limit-branch-curve}   
     B_0 = \lim_{t\to 0}\,B_t 
     = V\left(\lim_{t\to 0}\, t\star u\right) = V(\pi_+(u)+\pi_0(u))\subseteq\mathbb{P}(1,1,2)
\end{equation}
is obtained by composing $V(\cdot)$ with the morphism $\mathbb{P}(U_{\Sigma})_{\mathrm{reg}}\to \mathbb{P}(U_{\Sigma,+}\oplus U_{\Sigma,0})_{\mathrm{reg}}$ of Lemma~\ref{lem:decomposition-reg-elements}.
The central fiber $S_0$ is the double cover of $\mathbb{P}(1,1,2)$ with branch curve $B_0$.
\end{remark}

\begin{definition}
\label{def:sigma-generic}
Let $u\in U_\Sigma$ and consider the one-parameter family $\mathcal{S}(t\star u)\rightarrow\Delta$ in Definition~\ref{def:degenerations-Horikawa-we-compute-stable-replacement-of}. The central fiber $S_0\subseteq\mathcal{S}(t\star u)$ is given by $V(w^2-(\pi_0+\pi_+)(u))\subseteq\mathbb{P}(1,1,2,5)$. We say that $u$ is $\Sigma$-\emph{generic} provided the following hold:
\begin{enumerate}

\item The central fiber $S_0$ is singular only at $[1:0:0:0]$, where it has a singularity of type $\Sigma$;

\item The other fibers of $\mathcal{S}(t\star u)\rightarrow\Delta$ are smooth Horikawa surfaces in $\mathbb{P}(1,1,2,5)$.

\end{enumerate}
Such conditions are verified for a generic choice of the coefficients of the polynomial $u$. This can be observed at the level of the branch curve \eqref{eq:limit-branch-curve}, for which we need to show it has only one singularity of type $\Sigma$ at $[1:0:0]$. The idea is that for a generic $u$, the curve $V((\pi_0+\pi_+)(u))$ has a singularity of type $\Sigma$ at $[1:0:0]$, and one can choose a specific $u$ for which the corresponding curve only has a singularity of type $\Sigma$ at $[1:0:0]$. By the upper semicontinuity of Milnor numbers \cite[Theorem~2.6]{GLS07}, the generic $u$ has the claimed property. We illustrate an analogous argument in the proof of Theorem~\ref{prop:birational-equiv}. 
\end{definition}

\begin{remark}
There exists $r>0$ such that $S_t$ is smooth for $0<|t|<r$, i.e. $u$ is $\Sigma$-generic in the sense of the 
Definition \ref{def:sigma-generic}.
In particular, in order to be $\Sigma$-generic,
the coefficient of $z^5$ must be non-zero, otherwise
$S_0$ will also pass through the singular point 
$[0:0:1:0]\in\mathbb P(1,1,2,5)$.
\end{remark}

Finally, the following construction will be used in \S\,\ref{subsec:proof-of-semi-log-canonicity}.

\begin{definition}
\label{def:theta-to-describe-tail}
Let $\theta\colon\mathbb{C}[t,x,y,z]\to\mathbb{C}[t,\alpha,\beta]$ be the 
ring homomorphism defined by the rules $\theta(t)=t$, $\theta(x)=1$, $\theta(y)=\alpha$, 
$\theta(z)=\beta$. Let $[u]\in\mathbb{P}(U_\Sigma)_{\mathrm{reg}}$.
Then, $\theta((\pi_0+\pi_-)(t\star u))$ is a homogeneous polynomial of degree $d$ after assigning $t$ degree $1$, $\alpha$ degree $p$, and $\beta$ degree $q$. 
In this way, by composing $V(\cdot)$ with the morphism
$\mathbb P(U_{\Sigma})_{\mathrm{reg}}\to \mathbb P(U_{\Sigma,0}\oplus U_{\Sigma,-})_{\mathrm{reg}}$
produces a curve
\begin{equation}
\label{eq:tail-curve}
\mathcal{T}_0(u) = V(\theta((\pi_0+\pi_-)(t\star u)))
\end{equation}
of degree $d$ in $\mathbb P(1,p,q)$.
\end{definition}

\begin{example}
\label{ex:t-star-action-and-theta}
Let us choose $\Sigma=W_{12}$, so that $(p,q)=(4,5)$ and $d=20$. Let us construct an example of the family described together with the $\mathbb{C}^*$-action. The weight function is given by $\mathrm{wt}_\Sigma(x^ay^bz^c)=4b+5c-20$. So, we can consider
\[
u=z^5+x^5y^5+x^2z^4+x^{10},
\]
where the weights of the monomials are $5,0,0,-20$ respectively. We have that
\[
t\star u =z^5+x^5y^5+x^2z^4+t^{20}x^{10}.
\]
Therefore, in the central fiber for $t=0$, the limiting curve is given by the vanishing of $z^5+x^5y^5+x^2z^4=0$. Additionally, we have that
\[
\theta((\pi_0+\pi_-)(t\star u))([t:\alpha:\beta])
=\alpha^5+\beta^4+t^{20},
\]
which is homogeneous of degree $20$ in $\mathbb{P}(1,4,5)$ with coordinate $[t:\alpha:\beta]$.
\end{example}


\subsection{Stable replacement of the central fiber of the families}
\label{subsec:computation-stable-replacement}

Fix one of the eight singularity types $\Sigma$ and let $\mathcal{S}=\mathcal{S}(t\star u)\rightarrow\Delta$ be one of the families in Definition~\ref{def:degenerations-Horikawa-we-compute-stable-replacement-of} with $u$ $\Sigma$-generic. Then $\mathcal{S}\rightarrow\Delta$ comes equipped with a fiberwise $\mathbb{Z}_2$-action. The quotient by this action $\mathcal{S}\rightarrow\mathfrak{X}=\mathbb{P}(1,1,2)\times\Delta$ has branch divisor $\mathcal{B}\subseteq\mathfrak{X}$ which fiberwisely gives a curve of weighted degree $10$.

Let $\mathcal{S}'\rightarrow\Delta'$ denote the stable replacement of $\mathcal{S}\rightarrow\Delta$. As this is obtained after a combination of birational modifications of the central fiber and possibly after base changes branched at the origin of $\Delta$, then also $\mathcal{S}'\rightarrow\Delta'$ comes equipped with a fiberwise $\mathbb{Z}_2$-action away from the central fiber. It is a standard argument that this action extends to the whole $\mathcal{S}'$ (see for instance \cite[Lemma~3.14]{MS21}). Let $\mathcal{S}'\rightarrow\mathfrak{X}'$ be the quotient by this action and let $\mathfrak{B}'\subseteq\mathfrak{X}'$ be the branch locus. Then $\left(\mathfrak{X}',\frac{1}{2}\mathcal{B}'\right)\rightarrow\Delta'$ is also a family of stable pairs by the work of Alexeev--Pardini \cite{AP12}. In particular, the central fiber $S_0'\subseteq\mathcal{S}'$ is an appropriate double cover of $X_0'\subseteq\mathfrak{X}'$. From this discussion it follows that the first goal is to compute the stable replacement of the central fiber of $\left(\mathfrak{X},\frac{1}{2}\mathcal{B}\right)\rightarrow\Delta$.

\begin{definition}
Let $\mathfrak{X}'\rightarrow\mathfrak{X}$ be the weighted blow up of the central fiber $X_0\subseteq\mathfrak{X}$ at the point $\xi=[1:0:0]$ with weights $(p,q)$ according to the singularity type $\Sigma$ (see Table~\ref{tbl:weights-and-degrees-for-eight-singularities}). Denote by $Y\subseteq\mathfrak{X}'$ the exceptional divisor of the blow up, which is isomorphic to $\mathbb{P}(1,p,q)$. Let $Z$ be the strict transform of the central fiber $X_0\subseteq\mathfrak{X}$, and let $\mathcal{B}'\subseteq\mathfrak{X}'$ the strict transform of $\mathcal{B}$ with central fiber $B_0'$. Let $E$ be the exceptional divisor of $Z\rightarrow X_0$ and $G\subseteq\mathbb{P}(1,p,q)$ the curve $V(t)$, where we denote by $[t:\alpha:\beta]$ the coordinate of $\mathbb{P}(1,p,q)$.
\end{definition}

The first step is then to prove the following theorem.

\begin{theorem}
\label{stable-replacement-with-one-weighted-blow-up}
The central fiber $\left(X_0',\frac{1}{2}B_0'\right)$ of $\left(\mathfrak{X}',\frac{1}{2}\mathcal{B}'\right)\rightarrow\Delta$ is a stable pair. The pair $\left(X_0',\frac{1}{2}B_0'\right)$ is obtained by gluing
\[
\left(Y,G+\frac{1}{2}\mathcal{B}'|_Y\right),~\left(Z,E+\frac{1}{2}\mathcal{B}'|_Z\right),
\]
where $Y\cong\mathbb{P}(1,p,q)$ and $Z=\mathrm{Bl}_\xi^{(p,q)}\mathbb{P}(1,1,2)$ are glued along $G\cong\mathbb{P}^1\cong E$.
\end{theorem}

\begin{proof}
We first observe that $K_{X_0'}+\frac{1}{2}B_0'$ is $\mathbb{Q}$-Cartier: $B_0'$ is the restriction to $X_0'$ of $\mathcal{B}'$, which is Cartier, and $K_{X_0'}$ is also Cartier by applying the adjunction formula on the central fiber of the family $X_0'\subseteq\mathfrak{X}'$ (see \cite[Proposition~16.4]{Cor92}). Then, proving that $\left(X_0',\frac{1}{2}B_0'\right)$ is semi-log canonical boils down to show that $\left(Y,G+\frac{1}{2}\mathcal{B}'|_Y\right)$ and $\left(Z,E+\frac{1}{2}\mathcal{B}'|_Z\right)$ are log canonical. This is proved in Propositions~\ref{proof-that-weight-blow-up-is-semi-log-canonical-Y-side} and \ref{proof-that-weight-blow-up-is-semi-log-canonical-Z-side}. Finally, we check the ampleness of $K_Y+G+\frac{1}{2}\mathcal{B}'|_Y$ and $K_Z+E+\frac{1}{2}\mathcal{B}'|_Z$ in Propositions~\ref{Ampleness-on-tail-weighted-blow-up} and \ref{formulasforamplenessonweightedblowup}, which imply that $K_{X_0'}+\frac{1}{2}B_0'$ is ample. So, $\left(X_0',\frac{1}{2}B_0'\right)$ is a stable pair.
\end{proof}

\begin{theorem}
\label{thm:geometry-of-Y-and-Z}
The stable replacement $S_0'$ of the one-parameter family $\mathcal{S}=\mathcal{S}(t\star u)\rightarrow\Delta$ with $u$ $\Sigma$-generic is the gluing of a K3 surface $\widetilde{Y}$ with singularities of type $A_n$ with a surface $\widetilde{Z}$ satisfying $h^1(\mathcal{O}_{\widetilde{Z}})=0$, $h^2(\mathcal{O}_{\widetilde{Z}})=1$, and $K_{\widetilde{Z}}^2$ as given in Table~\ref{table:intro}. The gluing locus $\widetilde{Y}\cap\widetilde{Z}$ is isomorphic to $\mathbb{P}^1$. The components $\widetilde{Y}$ and $\widetilde{Z}$ have finite cyclic quotient singularities, and these are contained in the gluing locus. Moreover, the topological Euler characteristics of $\widetilde{Y}$ and $\widetilde{Z}$ are $\mu_\Sigma+3$ and $36-\mu_\Sigma$ respectively, where $\mu_\Sigma$ is the Milnor number of the singularity $\Sigma$.
\end{theorem}

\begin{proof}
By the discussion at the beginning of \S\,\ref{subsec:computation-stable-replacement}, the stable replacement $S_0'$ is obtained by taking an appropriate double cover of $X_0'$. In \S\,\ref{subsec:double-cover-of-Y} and \S\,\ref{subsec:double-cover-of-Z} we describe the double covers $\widetilde{Y}\rightarrow Y$ and $\widetilde{Z}\rightarrow Z$. These two surfaces are glued along $\widetilde{G}\rightarrow G$ and $\widetilde{E}\rightarrow E$, which are isomorphic to $\mathbb{P}^1$ as we prove in \S\,\ref{subsec:gluing-curve-deg-Horikawa-is-P1}. We prove that $\widetilde{Y}$ and $\widetilde{Z}$ have finite quotient singularities, which lie along the gluing locus, in Proposition~\ref{prop:Ytilde-and-Ztilde-have-quotient-singularities}. The claimed values of $h^1(\mathcal{O}_{\widetilde{Z}})$, $h^2(\mathcal{O}_{\widetilde{Z}})$, $K_{\widetilde{Z}}^2$, and of the topological Euler characteristic of $\widetilde{Z}$ are computed in Proposition~\ref{prop:double-cover-of-Z-invariants} and Corollary~\ref{cor:top-Eul-char-tildeZ}. The statement about the topological Euler characteristic of $\widetilde{Y}$ is proved in Corollary~\ref{cor:top-Euler-char-Ytilde}.
\end{proof}

\begin{remark}
\label{rmk:stable-replacement-at-level-of-Horikawa-surfaces}
Consider $\mathcal{S}=\mathcal{S}(t\star u)\rightarrow\Delta$ as in Definition~\ref{def:degenerations-Horikawa-we-compute-stable-replacement-of} with $u$ $\Sigma$-generic. One may want to describe the stable replacement directly on $\mathcal{S}$ and not passing through $\left(\mathfrak{X},\frac{1}{2}\mathcal{B}\right)\rightarrow\Delta$ as we described so far. This is done as follows. In the cases where $d$ is even, it is sufficient to take the weighted blow up of $\mathcal{S}$ at $(t,y,z,w)=(0,0,0,0)$ with weights $(1,p,q,d/2)$ in the affine patch $x\neq0$. If $d$ is odd, then we first have to perform the base change $\widetilde{\Delta}\rightarrow\Delta$ such that $s\mapsto t^2$ obtaining a new family $\widetilde{\mathcal{S}}\rightarrow\widetilde{\Delta}$, and then blow up $\widetilde{\mathcal{S}}$ at $(t,y,z,w)=(0,0,0,0)$ with weights $(1,2p,2q,d)$ in the affine patch $x\neq0$.
\end{remark}

Before we move on with the proofs of the above claims, we recall some preliminaries about weighted projective planes. A reference for the following well-known facts is \cite[\S\,5.1]{Has00}. Let $a,b$ be two positive coprime integers. Consider the weighted projective plane $\mathbb{P}(1,a,b)$ with coordinates $[x:y:z]$. Let $D_x=V(x)$, $D_y=V(y)$, $D_z=V(z)$. Then we have the following linear equivalences:
\[
abD_x\sim bD_y\sim aD_z.
\]
Moreover, $abD_x$ generates the Picard group of $\mathbb{P}(1,a,b)$, and the intersection numbers among $D_x,D_y,D_z$ are given by

\begin{center}
\renewcommand{\arraystretch}{1.4}
\begin{tabular}{|c|c|c|c|}
\hline
$\cdot$ & $D_x$ & $D_y$ & $D_z$
\\
\hline
$D_x$ & $1/ab$ & $1/b$ & $1/a$
\\
\hline
$D_y$ & $1/b$ & $a/b$ & $1$
\\
\hline
$D_z$ & $1/a$ & $1$ & $b/a$
\\
\hline
\end{tabular}
\end{center}

\begin{remark}
\label{C=10Dx}
Let $C\subseteq\mathbb{P}(1,1,2)$ be an irreducible curve of weighted degree $10$. As $\mathbb{P}(1,1,2)$ is $\mathbb{Q}$-factorial and since $2D_x$ generates $\mathrm{Pic}(\mathbb{P}(1,1,2))$, there exists a rational constant $c$ such that $C=cD_x$. Intersecting both sides with $D_z$, we obtain that $c=10$.
\end{remark}


\subsection{Proof of semi-log canonicity}
\label{subsec:proof-of-semi-log-canonicity}

Let us prove that the pair $\left(X_0',\frac{1}{2}B_0'\right)$ has semi-log canonical singularities. We already explained that $K_{X_0'}+\frac{1}{2}B_0'$ is $\mathbb{Q}$-Cartier, so we focus on showing that $\left(Y,G+\frac{1}{2}\mathcal{B}'|_Y\right)$ and $\left(Z,E+\frac{1}{2}\mathcal{B}'|_Z\right)$ are log canonical.

Let us first focus on the former pair. We have that $Y\cong\mathbb{P}(1,p,q)$ has coordinates $[t:\alpha:\beta]$ and $G=V(t)$. So, along $G$, $Y$ has a $\frac{1}{p}(1,q)$ singularity at $[0:1:0]$ and a $\frac{1}{q}(1,p)$ singularity at $[0:0:1]$. We now describe the curve $\mathcal{B}'|_Y$. For this we start with the equation of $\mathcal{B}\subseteq\mathfrak{X}$. By \eqref{eq:branch-curve-family}, this is $V(t\star u)$, where $u\in U_{\mathrm{reg}}$.  The inhomogeneous form of $t\star u$
in the affine patch $x=1$ is just $\theta(t\star u)$ (see Definition~\ref{def:theta-to-describe-tail}). By construction, 
$\theta(t\star u)\in\mathbb{C}[t,\alpha,\beta]$ is a sum of homogeneous polynomials of degree $\geq d$. Moreover, the degree $d$ part of $\theta(t\star u)$ is just $\theta((\pi_0+\pi_-)(t\star u))$. Therefore, using \eqref{eq:tail-curve},
\[
     \mathcal{B}'|_Y = V(\theta((\pi_0+\pi_-)(t\star u)))
     =\mathcal{T}_0(u).
\]

\begin{lemma}
\label{lem:tail-Y-describes-all-deg-d-curves-in-P(1,p,q)}
Let $W_\Sigma$ denote the vector space of homogeneous polynomials of degree $d$ in $\mathbb P(1,p,q)$. Then the morphism
\begin{align*}
\mathbb{C} m_1\oplus\mathbb{C}m_2\oplus U_{\Sigma,-}&\rightarrow W_\Sigma\\
u&\mapsto\theta(t\star u)
\end{align*}
is an isomorphism. Moreover, for generic $\omega\in W_\Sigma$ there exists
$\sigma\in\mathrm{Aut}(\mathbb P(1,p,q))$ such that 
$\sigma(\omega)\in\theta(U_{\Sigma,0}\oplus U_{\Sigma,-})$. 
\end{lemma}

\begin{proof}
In the following table we report for each singularity type the number of monomials of degree $d$ in $\mathbb{P}(1,p,q)$.
\begin{center}
\renewcommand{\arraystretch}{1.4}
\begin{tabular}{|c|c|c|c|c|c|c|c|c|}
\hline
Sing. & $E_{12}$ & $E_{13}$ & $E_{14}$ & $Z_{11}$ & $Z_{12}$ & $Z_{13}$ & $W_{12}$ & $W_{13}$ \\
\hline
\makecell{
\#monomials of
\\[1ex]
deg $d$ in $\mathbb{P}(1,p,q)$
} & $17$ & $18$ & $19$ & $15$ & $16$ & $17$ & $16$ & $17$ \\
\hline
\end{tabular}
\end{center}
This matches the number of degree $10$ monomials $x^ay^bz^c$ with $\mathrm{wt}_\Sigma(x^ay^bz^c)\leq0$, which were listed in Proposition~\ref{prop:classification-signs-of-deg-10-monomials}. So, to show the map in the statement is an isomorphism, it suffices to show it is surjective. An arbitrary degree $d$ monomial in $W_\Sigma$ is in the form $t^a\alpha^b\beta^c$ with $a+pb+qc=d$. This is the image of $x^{10-b-2c}y^bz^c$, which we need to show has non-positive weight. But this is true as $\mathrm{wt}_\Sigma(x^{10-b-2c}y^bz^c)=pb+qc-d=-a\leq0$. For
the statement about the existence of $\sigma\in\text{Aut}(\mathbb P(1,p,q))$, one simply checks case by case constructing an automorphism $\sigma$ which makes the nonzero coefficients of $\theta(m_1)$ and $\theta(m_2)$ equal, which is the required condition to be in $\theta(U_{\Sigma,0}\oplus U_{\Sigma,-})$. 
\end{proof}

In the next lemma we discuss the intersection $\mathcal{B}'|_Y\cap G$ depending on the singularity $\Sigma$.

\begin{lemma}
\label{lem:intersection-curve-tain-with-conductor}
The points in which $\mathcal{B}'|_Y$ intersects with $G$ are summarized in the table below.
\begin{center}
\renewcommand{\arraystretch}{1.4}
\begin{tabular}{|c|c|c|c|c|}
\hline
Sing. & $E_{12}$ & $E_{13}$ & $E_{14}$ & $Z_{11}$ \\
\hline
$\mathcal{B}'|_Y\cap G$ & $[0:1:-1]$ &
\makecell{
$[0:1:0]$,
\\[1ex]
$[0:-1:1]$
}
& $[0:1:-1]$ &
\makecell{
$[0:0:1]$,
\\[1ex]
$[0:1:-1]$
}
\\
\hline
Sing. & $Z_{12}$ & $Z_{13}$ & $W_{12}$ & $W_{13}$ \\
\hline
$\mathcal{B}'|_Y\cap G$ & \makecell{
$[0:1:0]$,
\\[1ex]
$[0:0:1]$,
\\[1ex]
$[0:-1:1]$
} &
\makecell{
$[0:0:1]$,
\\[1ex]
$[0:1:-1]$
} & $[0:-1:1]$ & \makecell{
$[0:1:0]$,
\\[1ex]
$[0:1:-1]$
} \\
\hline
\end{tabular}
\end{center}
\end{lemma}

\begin{proof}
The intersection of $\mathcal{B}'|_Y$ with $V(t)$ is 
given by $V(\theta(m_1+m_2))\cap V(t)$. Working case
by case, we see that 
$\theta(m_1+m_2) = \alpha^a\beta^b(\alpha^q+\beta^p)$.  The 
only solution to $\alpha^q+\beta^p=0$ and $t=0$ in $\mathbb P(1,p,q)$
is either $[0:1:-1]$ or $[0:-1:1]$ according to the table above. By inspection $(a,b)=(0,0)$ if 
$\Sigma=E_{12}$, $E_{14}$, $W_{12}$.   For $\Sigma=E_{13}$, $W_{13}$ we
have $(a,b)=(0,1)$.  For
$\Sigma=Z_{11}$, $Z_{13}$ we have $(a,b)=(1,0)$.  For $Z_{12}$ we have
$(a,b)=(1,1)$.
\end{proof}

\begin{remark}
Part of the information of Lemma~\ref{lem:intersection-curve-tain-with-conductor} can be found in Table~\ref{tbl:central-fiber-after-weighted-blow-up-depending-on-singularity}, where in each cell the triangle on the right represents $\left(Y,G+\frac{1}{2}\mathcal{B}'|_Y\right)$.
\end{remark}

\begin{proposition}
\label{proof-that-weight-blow-up-is-semi-log-canonical-Y-side}
The pair $\left(Y,G+\frac{1}{2}\mathcal{B}'|_Y\right)$ is log canonical.
\end{proposition}

\begin{proof}
By Lemma~\ref{lem:tail-Y-describes-all-deg-d-curves-in-P(1,p,q)} and by the genericity assumption on the curve $\mathcal{T}_0(u)$, we have that the curve $\mathcal{B}'|_Y$ is smooth away from $G$, which contains the singular points of $\mathbb{P}(1,p,q)$. So, we only have to check the log canonicity in a neighborhood of $G$. As illustrated in Lemma~\ref{lem:intersection-curve-tain-with-conductor}, $G$ and $\mathcal{B}'|_Y$ intersect transversely at a smooth point of $Y$, or at a singular torus fixed point of the toric variety $Y$. So, we have that $\left(Y,G+\frac{1}{2}\mathcal{B}'|_Y\right)$ is log canonical by combining \cite[Theorem~3.32]{Kol13} with \cite[Proposition~11.4.24~(a)]{CLS11}.
\end{proof}

We now focus on the other pair.

\begin{proposition}
\label{proof-that-weight-blow-up-is-semi-log-canonical-Z-side}
The pair $\left(Z,E+\frac{1}{2}\mathcal{B}'|_Z\right)$ is log canonical.
\end{proposition}

\begin{proof}
By the $\Sigma$-genericity assumption, the curve $\mathcal{B}'|_Z$ is smooth away from $E$. So, we only have to check log canonicity of the pair in a neighborhood of the exceptional divisor $E$. The blow up $Z=\mathrm{Bl}_\xi^{(p,q)}\mathbb{P}(1,1,2)$ may be singular along $E$ at the torus fixed points $t_1,t_2$. These singularities are toric, and dictated by the weights $p$ and $q$ (see the next Remark~\ref{rmk:singularities-of-Z-along-E}). As we discussed in the previous section, depending on the singularity $\Sigma$, the curve $\mathcal{B}'|_Z$ may pass through $t_1$ or $t_2$. So we argue again by cases. These are summarized in Table~\ref{tbl:central-fiber-after-weighted-blow-up-depending-on-singularity}. We can conclude that $\left(Z,E+\frac{1}{2}\mathcal{B}'|_Z\right)$ is log canonical again by combining \cite[Theorem~3.32]{Kol13} with \cite[Proposition~11.4.24~(a)]{CLS11}.
\end{proof}

\begin{remark}
\label{rmk:singularities-of-Z-along-E}
Let us compute the singularities of $Z$ along $E$. From a toric perspective, the cone in $\mathbb{R}^2$ corresponding to $\xi\in\mathbb{P}(1,1,2)$ is $C(\xi)= \langle (1,0), (0,1) \rangle $, where $\mathbb{R}_{\geq0}(1,0)$ corresponds to a torus fixed fiber of $\mathbb{P}(1,1,2)$ and hence $\mathbb{R}_{\geq0}(0,1)$ to a torus fixed section. By \cite[Remark after Proposition~4.4]{Has00} we know that $C(\xi)$ is subdivided by the ray $\mathbb{R}_{\geq0}(q,p)$ by the weighted blow up  (that is, the blow up of the ideal $(y^q,z^p)$). Therefore, let $t_1,t_2$ be the torus fixed points of $Z$ along $E$ with associated cones $\langle(1,0),(q,p)\rangle$ and $\langle(0,1),(q,p)\rangle$ respectively. By \cite[Proposition~10.1.2]{CLS11}, these give rise to cyclic quotient singularities of type $\frac{1}{p}(1,-q)$ and $\frac{1}{q}(1,-p)$ respectively.
\end{remark}

\begin{table}[hbtp]
\centering
\caption{Picture of the gluing of $\left(Z,E+\frac{1}{2}\mathcal{B}'|_{Z}\right)$ and $\left(Y,G+\frac{1}{2}\mathcal{B}'|_Y\right)$ along $G\cong E$ depending on the singularity $\Sigma$. If the divisor passes through a singular point, we report the singularity type with respect to $Z$ and $Y$ respectively. The top boundary point in $Y$ is $[0:1:0]$ and the bottom one is $[0:0:1]$.}
\label{tbl:central-fiber-after-weighted-blow-up-depending-on-singularity}
\begin{tabular}{|>{\centering\arraybackslash}m{3.7cm}|>{\centering\arraybackslash}m{3.7cm}|>{\centering\arraybackslash}m{3.7cm}|>{\centering\arraybackslash}m{3.7cm}|}
\hline
\begin{tikzpicture}[scale=0.61]

	\draw[line width=1pt] (0,0) -- (6,0);
	\draw[line width=1pt] (0,0) -- (0,3);
	\draw[line width=1pt] (0,3) -- (6,0);
	\draw[line width=1pt] (2,0) -- (2,2);

	\fill (2,0) circle (3pt);
	\fill (2,2) circle (3pt);

	\node at (5,3) {{\small$E_{12}$}};

    \draw[line width=1pt,red,rotate around={-25:(0,0)}] (1.6,1.3) [partial ellipse=40:170:1.5cm and 0.5cm];

\end{tikzpicture}
&
\begin{tikzpicture}[scale=0.61]

	\draw[line width=1pt] (0,0) -- (6,0);
	\draw[line width=1pt] (0,0) -- (0,3);
	\draw[line width=1pt] (0,3) -- (6,0);
	\draw[line width=1pt] (2,0) -- (2,2);

	\fill (2,0) circle (3pt);
	\fill (2,2) circle (3pt);

	\node at (3,2.3) {{\scriptsize$\frac{1}{2}(1,1)$}};
	\node at (1.2,1.6) {{\scriptsize$\frac{1}{2}(1,1)$}};

	\node at (5,3) {{\small$E_{13}$}};

    \draw[line width=1pt,red,rotate around={-18:(0,0)}] (0.25,1.85) [partial ellipse=-90:105:2.6cm and 0.7cm];

\end{tikzpicture}
&
\begin{tikzpicture}[scale=0.61]

	\draw[line width=1pt] (0,0) -- (6,0);
	\draw[line width=1pt] (0,0) -- (0,3);
	\draw[line width=1pt] (0,3) -- (6,0);
	\draw[line width=1pt] (2,0) -- (2,2);

	\fill (2,0) circle (3pt);
	\fill (2,2) circle (3pt);

	\node at (5,3) {{\small$E_{14}$}};

    \draw[line width=1pt,red,rotate around={-25:(0,0)}] (1.6,1.3) [partial ellipse=40:170:1.5cm and 0.5cm];

\end{tikzpicture}
&
\begin{tikzpicture}[scale=0.61]

	\draw[line width=1pt] (0,0) -- (6,0);
	\draw[line width=1pt] (0,0) -- (0,3);
	\draw[line width=1pt] (0,3) -- (6,0);
	\draw[line width=1pt] (2,0) -- (2,2);

	\fill (2,0) circle (3pt);
	\fill (2,2) circle (3pt);

	\node at (5,3) {{\small$Z_{11}$}};

	\node at (3.3,0.5) {{\scriptsize$\frac{1}{4}(1,3)$}};
	\node at (1.1,0.7) {{\scriptsize$\frac{1}{4}(1,1)$}};

    \draw[line width=1pt,red,rotate around={-25:(0,0)}] (0.70,1.2) [partial ellipse=-100:118:1.5cm and 0.5cm];

\end{tikzpicture}
\\
\hline
\begin{tikzpicture}[scale=0.61]

	\draw[line width=1pt] (0,0) -- (6,0);
	\draw[line width=1pt] (0,0) -- (0,3);
	\draw[line width=1pt] (0,3) -- (6,0);
	\draw[line width=1pt] (2,0) -- (2,2);

	\fill (2,0) circle (3pt);
	\fill (2,2) circle (3pt);

	\node at (3.1,-0.5) {{\scriptsize$\frac{1}{3}(1,2)$}};
	\node at (3.2,2.4) {{\scriptsize$\frac{1}{2}(1,1)$}};
	\node at (1,-0.5) {{\scriptsize$\frac{1}{3}(1,1)$}};
	\node at (1.5,3.1) {{\scriptsize$\frac{1}{2}(1,1)$}};

	\node at (5,3) {{\small$Z_{12}$}};

    \draw[line width=1pt,red,rotate around={-25:(0,0)}] (0.70,1.2) [partial ellipse=-100:59:1.5cm and 0.5cm];
    \draw[line width=1pt,red,rotate around={140:(0,0)}] (-1.3,-2.45) [partial ellipse=-110:80:1.5cm and 0.5cm];

\end{tikzpicture}
&
\begin{tikzpicture}[scale=0.61]

	\draw[line width=1pt] (0,0) -- (6,0);
	\draw[line width=1pt] (0,0) -- (0,3);
	\draw[line width=1pt] (0,3) -- (6,0);
	\draw[line width=1pt] (2,0) -- (2,2);

	\fill (2,0) circle (3pt);
	\fill (2,2) circle (3pt);

	\node at (5,3) {{\small$Z_{13}$}};

	\node at (3.3,0.5) {{\scriptsize$\frac{1}{5}(1,3)$}};
	\node at (1.1,0.7) {{\scriptsize$\frac{1}{5}(1,2)$}};

    \draw[line width=1pt,red,rotate around={-25:(0,0)}] (0.70,1.2) [partial ellipse=-100:118:1.5cm and 0.5cm];

\end{tikzpicture}
&
\begin{tikzpicture}[scale=0.61]

	\draw[line width=1pt] (0,0) -- (6,0);
	\draw[line width=1pt] (0,0) -- (0,3);
	\draw[line width=1pt] (0,3) -- (6,0);
	\draw[line width=1pt] (2,0) -- (2,2);

	\fill (2,0) circle (3pt);
	\fill (2,2) circle (3pt);

	\node at (5,3) {{\small$W_{12}$}};

    \draw[line width=1pt,red,rotate around={-25:(0,0)}] (1.6,1.3) [partial ellipse=40:170:1.5cm and 0.5cm];

\end{tikzpicture}
&
\begin{tikzpicture}[scale=0.61]

	\draw[line width=1pt] (0,0) -- (6,0);
	\draw[line width=1pt] (0,0) -- (0,3);
	\draw[line width=1pt] (0,3) -- (6,0);
	\draw[line width=1pt] (2,0) -- (2,2);

	\fill (2,0) circle (3pt);
	\fill (2,2) circle (3pt);

	\node at (3,2.3) {{\scriptsize$\frac{1}{3}(1,1)$}};
	\node at (1.2,1.6) {{\scriptsize$\frac{1}{3}(1,2)$}};

	\node at (5,3) {{\small$W_{13}$}};

    \draw[line width=1pt,red,rotate around={-18:(0,0)}] (0.25,1.85) [partial ellipse=-90:105:2.6cm and 0.7cm];

\end{tikzpicture}
\\
\hline
\end{tabular}
\end{table}


\subsection{Proof of ampleness}

We now show $K_{X_0'}+\frac{1}{2}B_0'$ is ample. This boils down to showing that $K_Y+G+\frac{1}{2}\mathcal{B}'|_Y$ and $K_Z+E+\frac{1}{2}\mathcal{B}'|_Z$ are ample.

\begin{proposition}
\label{Ampleness-on-tail-weighted-blow-up}
$K_Y+G+\frac{1}{2}\mathcal{B}'|_Y$ is ample.
\end{proposition}

\begin{proof}
Since $G$ generates $\mathrm{Pic}(Y)\otimes\mathbb{Q}$, there exists a rational constant $c$ such that $\mathcal{B}'|_Y\sim_\mathbb{Q}cG$. Additionally, we have that $D_{\alpha}+D_{\beta}\sim_\mathbb{Q}(p+q)G$. Therefore,
\[
K_Y+G+\frac{1}{2}\mathcal{B}'|_Y\sim_\mathbb{Q}-G-D_{\alpha}-D_{\beta}+G+\frac{c}{2}G\sim_\mathbb{Q}\left(\frac{c}{2}-p-q\right)G.
\]
So, $K_Y+G+\frac{1}{2}\mathcal{B}'|_Y$ is ample provided $\frac{c}{2}-p-q>0$. To compute the constant $c$, we intersect both sides of $\mathcal{B}'|_Y\sim_\mathbb{Q}cG$ with $G$ to obtain $\mathcal{B}'|_Y\cdot G=cG^2=\frac{c}{pq}$. Hence,
\[
c=pq(\mathcal{B}'|_Y\cdot G).
\]
To compute the intersection $\mathcal{B}'|_Y\cdot G$, we can use the calculations carried out in the proof of Lemma~\ref{lem:intersection-curve-tain-with-conductor} (these are visually summarized in Table~\ref{tbl:central-fiber-after-weighted-blow-up-depending-on-singularity}). For instance, for $\Sigma=E_{12}$, $\mathcal{B}'|_Y\cdot G=1$, hence $c=21$. Repeating this for each singularity we obtain the following table:
\begin{center}
\renewcommand{\arraystretch}{1.4}
\begin{tabular}{|c|c|c|c|c|c|c|c|c|}
\hline
Sing. & $E_{12}$ & $E_{13}$ & $E_{14}$ & $Z_{11}$ & $Z_{12}$ & $Z_{13}$ & $W_{12}$ & $W_{13}$ \\
\hline
$c$ & $21$ & $15$ & $24$ & $15$ & $11$ & $18$ & $20$ & $16$ \\
\hline
$\frac{c}{2}-p-q$ & $\frac{1}{2}$ & $\frac{1}{2}$ & $1$ & $\frac{1}{2}$ & $\frac{1}{2}$ & $1$ & $1$ & $1$ \\
\hline
\end{tabular}
\end{center}
The inequality $\frac{c}{2}-p-q>0$ is then verified by the above table (see Table~\ref{tbl:weights-and-degrees-for-eight-singularities} for the values of $p$ and $q$). In particular, $K_Y+G+\frac{1}{2}\mathcal{B}'|_Y$ is ample for each singularity type.
\end{proof}

\begin{proposition}
\label{formulasforamplenessonweightedblowup}
$K_Z+E+\frac{1}{2}\mathcal{B}'|_Z$ is ample.
\end{proposition}

\begin{proof}
We have that $B_0=10D_x$ by Remark~\ref{C=10Dx}. Recall that
\[
Z=\mathrm{Bl}_\xi^{(p,q)}X\rightarrow\mathbb{P}(1,1,2),
\]
so, in the affine patch $\{x\neq0\}$, we assign weight $p$ to $y$ and $q$ to $z$. For a divisor $D$ in $\mathbb{P}(1,1,2)$, let $\widehat{D}$ denote its strict transform. As $Z$ is a toric surface, we have that the divisor $K_Z+E+\frac{1}{2}\mathcal{B}'|_Z$ is ample if and only if it intersects each boundary curve $\widehat{D}_x,\widehat{D}_y,\widehat{D}_z,E$ positively. In what follows, we compute these four intersection numbers.

We have that $B_0\sim10D_x$, $D_y\sim D_x$, $D_z\sim2D_x$, and $\xi=[1:0:0]\notin D_x$. Then, if $\sigma\colon Z\rightarrow X$ denotes the weighted blow up at $\xi$, the following equalities hold:
\begin{align*}
10\widehat{D}_x&=\sigma^*(\mathcal{B}_0)=\mathcal{B}'|_{Z}+dE\implies\mathcal{B}'|_Z=10\widehat{D}_x-dE,\\
\widehat{D}_x&=\sigma^*D_y=\widehat{D}_y+pE\implies\widehat{D}_y=\widehat{D}_x-pE,\\
2\widehat{D}_x&=\sigma^*D_z=\widehat{D}_z+qE\implies\widehat{D}_z=2\widehat{D}_x-qE.
\end{align*}
Combining these equalities with $K_Z=-\widehat{D}_x-\widehat{D}_y-\widehat{D}_z-E$, we can rewrite
\[
K_Z+E+\frac{1}{2}\mathcal{B}'|_Z=\widehat{D}_x+\left(p+q-\frac{d}{2}\right)E.
\]
We can then compute the following intersection numbers:
\[
\widehat{D}_x\cdot\left(K_Z+E+\frac{1}{2}\mathcal{B}'|_Z\right)=\widehat{D}_x\cdot\left(\widehat{D}_x+\left(p+q-\frac{d}{2}\right)E\right)=\frac{1}{2},
\]

\begin{align*}
\widehat{D}_y\cdot\left(K_Z+E+\frac{1}{2}\mathcal{B}'|_Z\right)&=(\widehat{D}_x-pE)\cdot\left(\widehat{D}_x+\left(p+q-\frac{d}{2}\right)E\right)\\
&=\frac{1}{2}+\left(p+q-\frac{d}{2}\right)\frac{1}{q}=\frac{2p+3q-d}{2q},
\end{align*}

\begin{align*}
\widehat{D}_z\cdot\left(K_Z+E+\frac{1}{2}\mathcal{B}'|_Z\right)&=(2\widehat{D}_x-qE)\cdot\left(\widehat{D}_x+\left(p+q-\frac{d}{2}\right)E\right)\\
&=1+\left(p+q-\frac{d}{2}\right)\frac{1}{p}=\frac{4p+2q-d}{2p},
\end{align*}

\[
E\cdot\left(K_Z+E+\frac{1}{2}\mathcal{B}'|_Z\right)=E\cdot\left(\widehat{D}_x+\left(p+q-\frac{d}{2}\right)E\right)=\frac{-2p-2q+d}{2pq},
\]
where in the last equality we used that $E^2=-\frac{1}{pq}$ (see \cite[Theorem~4.3~(3)]{ABMMOG14}). The intersection with $D_x$ is $\frac{1}{2}$, independently of the singularity type. The intersections with $D_y,D_z,E$ are also positive, as shown in the table below as the singularity type varies.

\begin{center}
\renewcommand{\arraystretch}{1.4}
\begin{tabular}{|c|c|c|c|c|}
\hline
Sing. & $E_{12}$ & $E_{13}$ & $E_{14}$ & $Z_{11}$ \\
\hline
\makecell{
{\tiny$(K_Z+E+\mathcal{B}'|_Z)\cdot D$,}
\\[1ex]
{\tiny$D=\widehat{D}_y,\widehat{D}_z,E$}
} & $\frac{3}{7},~\frac{5}{6},~\frac{1}{42}$ & $\frac{2}{5},~\frac{3}{4},~\frac{1}{20}$ & $\frac{3}{8},~\frac{2}{3},~\frac{1}{24}$ & $\frac{3}{8},~\frac{5}{6},~\frac{1}{24}$ \\
\hline
\hline
Sing. & $Z_{12}$ & $Z_{13}$ & $W_{12}$ & $W_{13}$ \\
\hline
\makecell{
{\tiny$(K_Z+E+\mathcal{B}'|_Z)\cdot D$,}
\\[1ex]
{\tiny$D=\widehat{D}_y,\widehat{D}_z,E$}
} & $\frac{1}{3},~\frac{3}{4},~\frac{1}{12}$ & $\frac{3}{10},~\frac{2}{3},~\frac{1}{15}$ & $\frac{3}{10},~\frac{3}{4},~\frac{1}{20}$ & $\frac{1}{4},~\frac{2}{3},~\frac{1}{12}$ \\
\hline
\end{tabular}
\end{center}
\end{proof}


\subsection{The double cover \texorpdfstring{$\widetilde{Y}\rightarrow Y$}{Lg}}
\label{subsec:double-cover-of-Y}

If $d$ is even, we can form the double cover of 
$\mathbb{P}(1,p,q)$ with branch curve $\mathcal{B}'|_Y=\mathcal{T}_0(u)$ as the hypersurface of degree $d$ in $\mathbb P(1,p,q,d/2)$ given by $w^2=\theta((\pi_0+\pi_-)(t\star u))$. If $d$ is odd, we use the isomorphism $\mathbb P(1,p,q)\cong\mathbb P(1,2p,2q)$ to construct the double cover as a hypersurface in $\mathbb{P}(1,2p,2q,d)$. This amounts to replacing $t$ by $t^2$ and doubling the degrees of $\alpha$ and $\beta$. More geometrically, when $d$ is odd, one is constructing the double cover of $\mathbb{P}(1,p,q)$ branched along $\mathcal{T}_0(u)\cup V(t)$. We summarize this information in the table below.
\begin{center}
\renewcommand{\arraystretch}{1.4}
\begin{tabular}{|c|c|c|c|c|}
\hline
Sing. & $E_{12}$ & $E_{13}$ & $E_{14}$ & $Z_{11}$ \\
\hline
degree & $42$ & $30$ & $24$ & $30$ \\
\hline
Proj. space & $\mathbb{P}(1,6,14,21)$ & $\mathbb{P}(1,4,10,15)$ & $\mathbb{P}(1,3,8,12)$ & $\mathbb{P}(1,6,8,15)$ \\
\hline
\hline
Sing. & $Z_{12}$ & $Z_{13}$ & $W_{12}$ & $W_{13}$ \\
\hline
degree & $22$ & $18$ & $20$ & $16$ \\
\hline
Proj. space & $\mathbb{P}(1,4,6,11)$ & $\mathbb{P}(1,3,5,9)$ & $\mathbb{P}(1,4,5,10)$ & $\mathbb{P}(1,3,4,8)$ \\
\hline
\end{tabular}
\end{center}

From \cite[Lemma~7.1]{IF00} follows that $\widetilde{Y}$ is an ADE K3 surface. These are precisely eight of the $95$ families of codimension $1$ ADE K3 surfaces classified by Reid \cite{Rei79}. According to \cite[Table~1]{IF00}, these families are respectively No. $88,70,53,71,51,35,41,30$.

\begin{corollary}
\label{cor:top-Euler-char-Ytilde}
The topological Euler characteristic of $\widetilde{Y}$ equals $\mu_\Sigma+3$, where $\mu_\Sigma$ is the Milnor number of the surface singularity $\Sigma\in\{E_{12},E_{13},E_{14},Z_{11},Z_{12},Z_{13},W_{12},W_{13}\}$.
\end{corollary}
\begin{proof}
Let $\widehat{Y}\rightarrow\widetilde{Y}$ be the minimal resolution of singularities of $\widetilde{Y}$, which is a smooth K3 surface. Then, $\chi_{\mathrm{top}}(\widetilde{Y})$ equals $24$ minus the number of exceptional $\mathbb{P}^1$ in $\widetilde{Y}$. \cite[Table~1]{IF00} reports the ADE singularities that $\widetilde{Y}$ has. From this, one obtains the claim after checking case by case. For instance, if $\Sigma=E_{12}$, then $\widetilde{Y}$ has exactly three singular points, and these are $A_1,A_2,A_6$ singularities. So, $\chi_{\mathrm{top}}(\widetilde{Y})=24-1-2-6=15=\mu_{E_{12}}+3$.
\end{proof}


\subsection{The gluing curve \texorpdfstring{$\widetilde{Y}\cap\widetilde{Z}$}{Lg}}
\label{subsec:gluing-curve-deg-Horikawa-is-P1}

To describe the curve along which the stable surfaces $\widetilde{Y}$ and $\widetilde{Z}$ are glued, we view $\widetilde{G}:=\widetilde{Y}\cap\widetilde{Z}\subseteq\widetilde{Y}$ as the double cover of $G\subseteq Y$. The advantage is that for $\widetilde{G}$ we have an explicit equation
\begin{equation}
\label{eq:gluing-curve-degenerate-cover}
w^2=\theta(\pi_0(u))
\end{equation}
in $\mathbb{P}(p,q,d/2)$ or $\mathbb{P}(2p,2q,d)$, depending whether $d$ is even or odd respectively. We will prove that $\widetilde{G}$ is isomorphic to $\mathbb{P}^1$.

If the singularity type is $E_{12},E_{13},Z_{11}$, or $Z_{12}$ (which correspond to $d$ odd), then the isomorphism $\mathbb{P}(2p,2q,d)\cong\mathbb{P}(p,q,d)$ induces an isomorphism of the curve \eqref{eq:gluing-curve-degenerate-cover} with $w=\theta(\pi_0(u))$. This proves that $G$ is in the branch locus of the cover $\widetilde{Y}\rightarrow Y$, as each point on it has only one preimage. In particular, $\widetilde{G}\cong\mathbb{P}^1$.

We now analyze the case of $E_{14},Z_{13},W_{12},W_{13}$ (which correspond to $d$ even). In these cases, the restriction to $\widetilde{G}$ of the projection $\mathbb{P}(p,q,d/2)\dashrightarrow\mathbb{P}(p,q)$ such that $[\alpha:\beta:w]\mapsto[\alpha:\beta]$ gives a $2:1$ morphism branched at two distinct points (this can be checked inspecting the four cases). In conclusion, $\widetilde{G}$ is isomorphic to $\mathbb{P}^1$ also if $d$ is even.
We illustrate this strategy with one of the cases, since the other ones are analogous.  For $E_{14}$ and under the isomorphisms $\mathbb{P}(3,8,12)\cong\mathbb{P}(3,2,3)\cong\mathbb{P}(1,2,1)$, the curve $\widetilde{G}$ becomes identified with
\[
\widetilde{G}=\{w^2=\alpha^8+\beta^3\}\cong\{w^2=\alpha^2+\beta^3\}\cong\{w^2=\alpha^2+\beta\}=C.
\]
The restriction to $C$ of the projection $\mathbb{P}(1,2,1)\dashrightarrow\mathbb{P}(1,2)$ such that $[\alpha:\beta:w]\mapsto[\alpha:\beta]$ is $2:1$ and it is branched at the points $[0:1]$ and $[1:-1]$. So $\widetilde{G}\cong C\cong\mathbb{P}^1$.


\subsection{The double cover \texorpdfstring{$\widetilde{Z}\rightarrow Z$}{Lg}}
\label{subsec:double-cover-of-Z}

We now study the geometry of the double cover $\widetilde{Z}\rightarrow Z=\mathrm{Bl}_\xi^{(p,q)}\mathbb{P}(1,1,2)$.

\begin{proposition}
\label{prop:double-cover-of-Z-invariants}
The double cover $\widetilde{Z}\rightarrow Z$ satisfies $h^1(\mathcal{O}_{\widetilde{Z}})=0$, $h^2(\mathcal{O}_{\widetilde{Z}})=1$. Moreover, $K_{\widetilde{Z}}^2$ is given as follows:
\begin{center}
\renewcommand{\arraystretch}{1.4}
\begin{tabular}{|c|c|c|c|c|c|c|c|c|}
\hline
Sing. & $E_{12}$ & $E_{13}$ & $E_{14}$ & $Z_{11}$ & $Z_{12}$ & $Z_{13}$ & $W_{12}$ & $W_{13}$ \\
\hline
$K_{\widetilde{Z}}^2$ & $\frac{19}{21}$ & $\frac{4}{5}$ & $\frac{2}{3}$ & $\frac{5}{6}$ & $\frac{2}{3}$ & $\frac{7}{15}$ & $\frac{3}{5}$ & $\frac{1}{3}$ \\
\hline
\end{tabular}
\end{center}
\end{proposition}

\begin{proof}
Recall from \S\,\ref{subsec:gluing-curve-deg-Horikawa-is-P1} that $\widetilde{Y}\cap\widetilde{Z}$ is part of the ramification divisor if and only if $d$ is odd. Therefore, we distinguish two cases.

If $d$ is even, the branch divisor of $\widetilde{Z}\rightarrow Z$ equals $\mathcal{B}'|_Z$. Using the expressions for $\mathcal{B}'|_Z,\widehat{D}_y,\widehat{D}_z$ computed in the proof of Proposition~\ref{formulasforamplenessonweightedblowup}, we obtain that
\begin{align*}
K_{\widetilde{Z}}&=\pi^*\left(K_Z+\frac{1}{2}\mathcal{B}'|_Z\right)\\
&=\pi^*\left(-\widehat{D}_x-\widehat{D}_y-\widehat{D}_z-E+5\widehat{D}_x-\frac{d}{2}E\right)\\
&=\pi^*\left(\widehat{D}_x+\left(p+q-1-\frac{d}{2}\right)E\right)\\
\implies K_{\widetilde{Z}}^2&=2\left(\widehat{D}_x+\left(p+q-1-\frac{d}{2}\right)E\right)^2=1-\frac{1}{2pq}(2p+2q-2-d)^2,
\end{align*}
from which we obtain the claimed values of $K_{\widetilde{Z}}^2$. To compute the cohomology of $\mathcal{O}_{\widetilde{Z}}$ we use that $\pi_*\mathcal{O}_{\widetilde{Z}}\cong\mathcal{O}_Z\oplus\mathcal{O}_Z(\frac{d}{2}E-5\widehat{D}_x)$. We have that $h^1(\mathcal{O}_Z)=h^2(\mathcal{O}_Z)=0$ because $Z$ is a rational surface with rational singularities. Hence, we obtain that
\[
h^1(\mathcal{O}_{\widetilde{Z}})=h^1\left(\frac{d}{2}E-5\widehat{D}_x\right)=0,~h^2(\mathcal{O}_{\widetilde{Z}})=h^2\left(\frac{d}{2}E-5\widehat{D}_x\right)=1,
\]
where we used \cite[Proposition~9.1.6]{CLS11}. Alternatively, one can use the following Macaulay2 code:

\begin{verbatim}
i1: loadPackage "NormalToricVarieties";
i2: rayList = {{1,0},{q,p},{0,1},{-1,-2}};
i3: coneList = {{0,1},{1,2},{2,3},{3,0}};
i4: WX = normalToricVariety(rayList,coneList);
i5: D = toricDivisor({0,d/2,0,-5},WX);
i6: SD = OO D;
i7: {HH^1(WX,SD),HH^2(WX,SD)}
\end{verbatim}

If $d$ is odd, the branch divisor of $\widetilde{Z}\rightarrow Z$ equals $\mathcal{B}'|_Z+E$ instead. Hence,
\begin{align*}
K_{\widetilde{Z}}&=\pi^*\left(K_Z+\frac{1}{2}(\mathcal{B}'|_Z+E)\right)\\
&=\pi^*\left(-\widehat{D}_x-\widehat{D}_y-\widehat{D}_z-E+5\widehat{D}_x+\frac{1-d}{2}E\right)\\
&=\pi^*\left(\widehat{D}_x+\left(p+q-\frac{1+d}{2}\right)E\right)\\
\implies K_{\widetilde{Z}}^2&=2\left(\widehat{D}_x+\left(p+q-\frac{1+d}{2}\right)E\right)^2=1-\frac{1}{2pq}(2p+2q-1-d)^2,
\end{align*}
from which we obtain the remaining values of $K_{\widetilde{Z}}^2$ in the table. For the cohomology of $\mathcal{O}_{\widetilde{Z}}$ we use that $\pi_*\mathcal{O}_{\widetilde{Z}}\cong\mathcal{O}_Z\oplus\mathcal{O}_Z(\frac{d-1}{2}E-5\widehat{D}_x)$. As $h^1(\mathcal{O}_Z)=h^2(\mathcal{O}_Z)=0$, we obtain that
\[
h^1(\mathcal{O}_{\widetilde{Z}})=h^1\left(\frac{d-1}{2}E-5\widehat{D}_x\right)=0,~h^2(\mathcal{O}_{\widetilde{Z}})=h^2\left(\frac{d-1}{2}E-5\widehat{D}_x\right)=1,
\]
by \cite[Proposition~9.1.6]{CLS11}, or by the same Macaulay2 code as above with $d$ replaced by $d-1$.
\end{proof}

Next, we compute the topological Euler characteristic of $\widetilde{Z}$ across the eight singularity types. Preliminarily, we find the Euler characteristic of the singular curve $B_0\subseteq\mathbb{P}(1,1,2)$ in \eqref{eq:limit-branch-curve}. To do this, we start by recalling the following geometric genus formula for plane curves 
(see \cite{CAMMOG14}): Let $D\subseteq\mathbb P^2$ be a smooth curve of degree $d$ and $C\subseteq\mathbb{P}^2$ be an integral curve of degree $d$ with normalization $\pi\colon\widehat{C}\to C$. Then,
\begin{equation}
\label{eq:genus-curve-in-weighted-proj-plane}
g(\widehat{C}) = g(D) - \sum_{p\in\mathrm{sing}(C)}\delta_p,
\end{equation}
where
\begin{itemize}
\item $g(E)$ is the genus of the curve $E$;

\item $2\delta_p = \mu_p + |\pi^{-1}(p)|-1$;

\item $\mu_p$ is the Milnor number of $C$ at $p$.

\end{itemize}

\par Let $\mathbb{P}_{\omega}=\mathbb{P}(w_0,w_1,w_2)$ where $\gcd(w_i,w_j)=1$ for $i\neq j$. Assume that 
$\mathbb{P}_{\omega}$ contains a smooth curve of $D$ degree $d$. 
Then, by \cite[Theorem 5.6]{CAMMOG14}, the geometric genus formula \eqref{eq:genus-curve-in-weighted-proj-plane} holds for any integral curve $C\subseteq\mathbb{P}_{\omega}$ of degree 
$d$ which does not pass through a singular point of $\mathbb{P}_{\omega}$. 
Moreover, by \cite[Corollary~5.4]{CAMMOG14}, in this case 
$$
         g(D)=\frac{d(d-w_0-w_1-w_2)}{2w_0w_1w_2}+1
$$
which simplifies to the genus-degree formula for curves in $\mathbb P^2$, but need not be an integer if there are no smooth curves  in
$\mathbb P_{\omega}$.

\begin{lemma}
\label{lemma:EulerCharB0}
The topological Euler characteristic of the curve $B_0$ with singularity $\Sigma$ at $\xi=[1:0:0]$ is given by $\mu_\Sigma-30$, where $\mu_\Sigma$ is the Milnor number of $B_0$ at $\xi$.
\end{lemma}

\begin{proof}
Consider the short exact sequence (see \cite[\S\,5]{Dur81})
\begin{equation}
\label{eq:limit-curve-mhs}
0\to H^0(\{\xi\})\to H^0(\pi^{-1}(\xi))\to H^1(B_0)\to H^1(\widehat{B}_0)\to 0.
\end{equation}
This implies that
\begin{align}
\begin{split}
\label{eq:rk-of-H1-ofB0}
\mathrm{rk}\,H^1(B_0)&=\mathrm{rk}\,H^1(\widehat{B}_0)+\mathrm{rk}\,H^0(\pi^{-1}(\xi))-\mathrm{rk}\,H^0(\{\xi\})\\
&=2g(\widehat{B}_0)+|\pi^{-1}(\xi)|-1.
\end{split}
\end{align}
By \eqref{eq:genus-curve-in-weighted-proj-plane}, we have that $g(\widehat{B}_0)=16-\delta_\xi$. Here we chose as $D$ the Fermat curve of degree $10$ in $\mathbb{P}(1,1,2)$, which is smooth of genus $g(D)=16$, and we used that the generic curve $B_0$ is singular only at the point $\xi=[1:0:0]$, which is not an orbifold point of $\mathbb{P}(1,1,2)$. By substituting this in \eqref{eq:rk-of-H1-ofB0} together with the equality $2\delta_p = \mu_p + |\pi^{-1}(p)|-1$, we obtain
\[
\mathrm{rk}\,H^1(B_0)=32-\mu_\Sigma.
\]
Therefore, $\chi_{\mathrm{top}}(B_0)=\mu_\Sigma-30$.
\end{proof}

\begin{remark}
The sequence \eqref{eq:limit-curve-mhs}
is an exact sequence of mixed Hodge structure in which 
every term except $H^1(B_0)$ is pure of weight equal the cohomological degree.  Consequently,
$\mathrm{Gr}^W_1 H^1(B_0)\cong H^1(\widehat{B}_0)$ whereas 
$W_0 H^1(B_0)$ has rank $|\pi^{-1}(\xi)|-1$.   Moreover, 
the normalization $\pi\colon\widehat{B}_0\to B_0$
in this case is given by the strict transform of $B_0$ relative to the weighted blow up $Z\rightarrow\mathrm{Bl}_\xi^{(p,q)}\mathbb{P}(1,1,2)$, and hence $|\pi^{-1}(\xi)|$ is
just the number of times the red curve intersects the exceptional divisor in Table~\ref{tbl:central-fiber-after-weighted-blow-up-depending-on-singularity}.
\end{remark}

\begin{corollary}
\label{cor:top-Eul-char-tildeZ}
The topological Euler characteristic of the surface $\widetilde{Z}$ is $36-\mu_\Sigma$, where $\mu_\Sigma$ is the Milnor number of the singularity $\Sigma$.
\end{corollary}

\begin{proof}
The surface $S_0$ is the double cover of $\mathbb{P}(1,1,2)$ branched along $B_0\cup\{\zeta\}$, where $\zeta=[0:0:1]$ is the singular point of $\mathbb{P}(1,1,2)$. Therefore,
\begin{align*}
\chi_{\mathrm{top}}(S_0)&=2\chi_{\mathrm{top}}(\mathbb{P}(1,1,2)\setminus (B_0\cup\{\zeta\}))+\chi_{\mathrm{top}}(B_0)+1\\
&=2\chi_{\mathrm{top}}(\mathbb{P}(1,1,2))-\chi_{\mathrm{top}}(B_0)-1=6-(\mu_\Sigma-30)-1=35-\mu_\Sigma,
\end{align*}
where we used Lemma~\ref{lemma:EulerCharB0} for $\chi_{\mathrm{top}}(B_0)$. As $\widetilde{Z}$ is a weighted blow up with exceptional divisor $E\cong\mathbb{P}^1$ of $S_0$ at a single point, using again the additivity of the topological Euler characteristic we obtain that $\chi_{\mathrm{top}}(\widetilde{Z})=\chi_{\mathrm{top}}(S_0)+1=36-\mu_\Sigma$.
\end{proof}


\subsection{The singularities of \texorpdfstring{$\widetilde{Y}$}{Lg} and \texorpdfstring{$\widetilde{Z}$}{Lg}}
\label{subsec:Ytilde-and-Ztilde-have-quotient-singularities}

We conclude the proof of Theorem~\ref{thm:geometry-of-Y-and-Z} with the following proposition.

\begin{proposition}
\label{prop:Ytilde-and-Ztilde-have-quotient-singularities}
Across the eight singularity types, the surfaces $\widetilde{Y},\widetilde{Z}$ only have finite cyclic quotient singularities along $\widetilde{Y}\cap\widetilde{Z}$.
\end{proposition}

\begin{proof}
Consider the curves $\widetilde{G}\subseteq\widetilde{Y}$ and $\widetilde{E}\subseteq\widetilde{Z}$, which are double covers of $G\subseteq Y$ and $E\subseteq Z$ respectively. From the work we carried out so far, we know that the pairs $(\widetilde{Y},\widetilde{G})$ and $(\widetilde{Z},\widetilde{E})$ are log canonical. The singularities of $\widetilde{Y}$ and $\widetilde{Z}$ only occur along the gluing curves $\widetilde{G}$ and $\widetilde{Z}$. As the pairs $(\widetilde{Y},\widetilde{G})$ and $(\widetilde{Z},\widetilde{E})$ are log canonical, we can apply \cite[Lemma~5.5]{Ish00} to conclude that the isolated singularities of $\widetilde{Y}$ and $\widetilde{Z}$ 
along the respective gluing loci are log terminal singularities.  Furthermore, by \cite[\S\,4]{KSB88}, we know these singularities are cyclic quotient ones. 
\end{proof}


\subsection{Summary of the construction of the stable surface \texorpdfstring{$\widetilde{Y}\cup\widetilde{Z}$}{Lg}}

The goal of this subsection is to summarize the construction of the limit surfaces $\widetilde{Y}\cup\widetilde{Z}$ described so far. Along the way, we use the case of $W_{12}$ as a guiding example. Let $[x:y:z]$ be the coordinate of $\mathbb{P}(1,1,2)$. Let $\mathbb{V}_{10}$ denote the complex vector space of degree $10$ polynomials in $x,y,z$. Let $\mathbb{M}$ the basis of $\mathbb{V}_{10}$ consisting of the possible degree $10$ monomials. For each singularity type $\Sigma$ considered above, let $(p,q),d$ as in Table~\ref{tbl:weights-and-degrees-for-eight-singularities}. 
For $W_{12}$, we have $(p,q)=(4,5)$ and $d=20$. Consider the weight function $\mathrm{wt}_{\Sigma}(x^ay^bz^c)=pb+qc-d$. Let $m_1,m_2\in\mathbb{M}$ be the only two monomials of weight $0$. 
For $W_{12}$ these are $x^5y^5,x^2z^4$. Let $U_\Sigma$ denote the subspace of $\mathbb{V}_{10}$ consisting of elements such that $m_1$ and $m_2$ have the same coefficient. Given a $\Sigma$-generic $u\in U_{\Sigma}$ in the sense of Definition~\ref{def:sigma-generic}, decompose it as $u=\pi_-(u)+\pi_0(u)+\pi_+(u)$, see Definition~\ref{def:weighted-subspaces-and-projections} --- This notation is nothing more than the monomials of negative, zero, and positive degree with respect to the weight functions.  

We consider the one-parameter family $\mathcal{S}(t\star u)\rightarrow\Delta$ (see Definitions~\ref{def:t-star-action} and \ref{def:degenerations-Horikawa-we-compute-stable-replacement-of}), of which we want to compute the stable replacement of the central fiber, which is given by $\widetilde{Y}(u)\cup\widetilde{Z}(u)$.

\begin{itemize}

\item If $d$ is even consider $\mathbb{P}(1,p,q,d/2)$ with coordinate $[t:\alpha:\beta:w]$. Let $\widetilde{Y}(u)$ is the hypersurface of degree $d$ in $\mathbb{P}(1,p,q,d/2)$ given by
the polynomial equation
\[
w^2=\theta((\pi_-+\pi_0)(t\star u))\left([t:\alpha:\beta] \right),
\]
where $\theta$ was introduced in Definition~\ref{def:theta-to-describe-tail}. For an example in the case of $W_{12}$, see Example~\ref{ex:t-star-action-and-theta}. If $d$ is odd consider instead $\mathbb{P}(1,2p,2q,d)$ again with coordinate $[t:\alpha:\beta:w]$. Let $\widetilde{Y}(u)$ be the hypersurface of degree $2d$ in $\mathbb{P}(1,2p,2q,d)$ given by
\[
w^2=\theta((\pi_-+\pi_0)(t^2\star u))\left([t: \alpha : \beta] \right).
\]
In either case, let $\widetilde{G}\subseteq\widetilde{Y}(u)$ be the curve $V(t)\cap\widetilde{Y}(u)$, which is isomorphic to $\mathbb{P}^1$ as shown in \S\,\ref{subsec:gluing-curve-deg-Horikawa-is-P1}.

\item The surface $\widetilde{Z}(u)$ is the double cover of $\mathrm{Bl}_\xi^{(p,q)}\mathbb{P}(1,1,2)$, where $\xi=[1:0:0]$, with branch curve $V((\pi_0+\pi_+)(u))$ if $d$ is even, or $V((\pi_0+\pi_+)(u))$ union the exceptional divisor of the blow up $E\subseteq\mathrm{Bl}_\xi^{(p,q)}\mathbb{P}(1,1,2)$ if $d$ is odd. Let $\widetilde{E}\subseteq\widetilde{Z}$ be the preimage of $E$.

\end{itemize}

The surfaces $\widetilde{Y}(u)$ and $\widetilde{Z}(u)$ are glued along the curves $\widetilde{G}\cong\mathbb{P}^1\cong\widetilde{E}$.


\subsection{One-parameter degenerations over a DVR}
\label{KSBA-DVR}

\par Thus far, we have computed the stable replacement of the central fiber of the families $\mathcal{S}(t\star u)\rightarrow\Delta$ described in  Definition~\ref{def:degenerations-Horikawa-we-compute-stable-replacement-of}. Such stable replacement can be understood as in Remark~\ref{rmk:stable-replacement-at-level-of-Horikawa-surfaces}. Although this is not going to be used later in the paper, we point out that this actually extends to other more general families.

Suppose that $R$ is a DVR with residue field $\mathbb{C}$ 
and $\mathbb{D}=\mathrm{Spec}(R)$ with uniformizing parameter $s$. In this section, 
we explain how to modify our previous work to determine the KSBA stable replacement for one-parameter degenerations of Horikawa surfaces of type $\Sigma$ over 
$\mathbb{D}$. For simplicity of exposition, we focus on the case where 
$\Delta\subseteq\mathbb{C}$ is a disk and $R=\mathcal{O}_p$ is the ring of germs of holomorphic functions at $p\in\Delta$.

\par Given $g\in\mathcal{O}_p$, let $k=\mathrm{ord}(g)$ denote the order of vanishing of $g$ at $p$ and let
\[
     \tau(g) = g^{(k)}(0)s^k/k!
\]
be the truncation of $g$ to its lowest order term.  Analogously, the truncation of 
$f\in\mathbb{V}_{10}\otimes\mathcal{O}_p$ is defined component by component relative to the basis $\mathbb{M}$ of degree $10$ monomials in $\mathbb{P}(1,1,2)$, i.e.
\begin{equation*}
    f=\sum_{m\in\mathbb{M}}\, f_m(s)m \implies \tau(f) = \sum_{m\in\mathbb{M}}\, \tau(f_m)m.
\end{equation*}

\begin{definition}
Given a singularity type $\Sigma$, we say that an element
$f\in U_{\Sigma}\otimes\mathcal{O}_p$ is \emph{$\Sigma$-generic} if 
there exists $u\in(U_{\Sigma})_{\mathrm{reg}}$ which is $\Sigma$-generic and satisfying $\tau(f) = s\star u$.
\end{definition}

If $f\in U_\Sigma\otimes\mathcal{O}_p$ is $\Sigma$-generic, then
\begin{equation*}
       \mathcal{S}(f) := \{([x:y:z:w],q)\in\mathbb P(1,1,2,5)\times\Delta \mid
                         w^2 - f(s(q))=0\}
\end{equation*}
is a one-parameter degeneration of smooth Horikawa surfaces. It turns out that the stable replacement of the central fiber of $\mathcal{S}(f)$ is computed 
in the same way as for the degenerations in
Definition~\ref{def:degenerations-Horikawa-we-compute-stable-replacement-of}. Before proving this we first introduce the following notation. We denote by $\widetilde{\mathcal{S}}(f)$ the base change of $\mathcal{S}(f)$ with respect to $s\mapsto s^2$.

\begin{proposition}
\label{prop:DVR-stable-replacement}
Let $f$ be $\Sigma$-generic. Consider the one-parameter families $\mathcal{S}(f)$ and $\mathcal{S}(\tau(f))$. As constructed in Remark~\ref{rmk:stable-replacement-at-level-of-Horikawa-surfaces}, consider the following modified families:
\begin{itemize}

\item Assume $d$ is even. Let $\mathcal{S}(f)'$ and $\mathcal{S}(\tau(f))'$ be respectively the weighted blow up of $\mathcal{S}(f)$ and $\mathcal{S}(\tau(f))$ with respect to the ideal $(s^d,y^q,z^p,w^2)$.

\item Assume $d$ is odd. Let $\mathcal{S}(f)'$ and $\mathcal{S}(\tau(f))'$ be respectively the weighted blow up of $\widetilde{\mathcal{S}}(f)$ and $\widetilde{\mathcal{S}}(\tau(f))$ with respect to the ideal $(s^{2d},y^{2q},z^{2p},w^2)$.

\end{itemize}
Then, the central fibers of $\mathcal{S}(f)'$ and $\mathcal{S}(\tau(f))'$ are isomorphic. In particular, $\mathcal{S}(f)'\rightarrow\Delta$ provides the stable replacement of the central fiber of $\mathcal{S}(f)\rightarrow\Delta$.
\end{proposition}

\begin{proof}
Let $\widetilde{Z}_\tau\cup\widetilde{Y}_\tau$ and $\widetilde{Z}\cup\widetilde{Y}$ be the central fibers of $\mathcal{S}(\tau(f))$ and $\mathcal{S}(f)$ respectively, where $\widetilde{Y}_\tau,\widetilde{Y}$ denote the exceptional divisors. We already know that $\widetilde{Z}_\tau\cup\widetilde{Y}_\tau$ is a stable surface by the discussion in \S\,\ref{subsec:computation-stable-replacement} (see in particular Remark~\ref{rmk:stable-replacement-at-level-of-Horikawa-surfaces}).

As $V(\lim_{s\to0}\tau(f))=V(\lim_{s\to0}f)=:C$, then we have that $\widetilde{Z}_\tau$ and $\widetilde{Z}$ are isomorphic because they are both the double cover of $\mathrm{Bl}_{\xi}^{(p,q)}\mathbb{P}(1,1,2)$ with branch divisor the strict transform of the curve $C$ (union the exceptional divisor if $d$ is odd).

Let $\phi\colon\mathbb{C}[[s]][x,y,z]\to\mathbb{C}[[t]][x,y,z]$ denote the map obtained by 
setting $\phi(s)=t,\phi(x)=1,\phi(y)=u$, and $\phi(z)=v$ ($\phi$ is the analogue of $\theta$ in Definition~\ref{def:theta-to-describe-tail}). We distinguish two cases. If $d$ is even, then $Y_\tau$ is given by $V((\pi_0+\pi_-)(\phi(\tau(f))))\subseteq\mathbb{P}(1,p,q,d/2)$. On the other hand, $\widetilde{Y}$ is given by the vanishing of the lower degree part of $(\pi_0+\pi_-)(\phi(f))$, which is precisely $(\pi_0+\pi_-)(\phi(\tau(f)))$, so $\widetilde{Y}_\tau$ and $\widetilde{Y}$ coincide. If $d$ is odd, then the argument is analogous to the previous one, with the difference that $\widetilde{Y}_\tau$ and $\widetilde{Y}$ are hypersurfaces in the weighted projective space $\mathbb{P}(1,2p,2q,d)$.
\end{proof}

\section{Dimension count of boundary strata}
\label{sec:dim-count-for-boundary-strata}

We now use the families in Definition~\ref{def:degenerations-Horikawa-we-compute-stable-replacement-of} to define eight closed and irreducible subsets of the boundary of $\overline{\mathbf{M}}$, one for each singularity type $\Sigma$. The starting point is the construction of a family of degenerate stable surfaces over $\mathbb{P}(U_\Sigma)_{\mathrm{reg}}$.

\begin{definition}
\label{def:the-eight-boundary-divisors}
In a slight abuse of notation with our conventions so far, we let $\Delta=\mathrm{Spec}(\mathbb{C}[[t]])$. Let
\[
\mathfrak{F}:=\mathbb{P}(U_\Sigma)_{\mathrm{reg}}\times\Delta\times\mathbb{P}(1,1,2)\rightarrow\mathbb{P}(U_\Sigma)_{\mathrm{reg}}\times\Delta
\]
be the usual projection map. Let $u\in\mathbb{P}(U_\Sigma)_{\mathrm{reg}}$. Define $\mathfrak{D}\subseteq\mathfrak{F}$ to be the closed subset given by $V(t\star u)$. Therefore, the fiber of $\left(\mathfrak{F},\frac{1}{2}\mathfrak{D}\right)\rightarrow\mathbb{P}(U_\Sigma)_{\mathrm{reg}}\times\Delta$ over $(u,t)\in\mathbb{P}(U_\Sigma)_{\mathrm{reg}}\times\Delta$ is $\left(\mathbb{P}(1,1,2),\frac{1}{2}\mathfrak{D}_{(u,t)}\right)$, where the curve has a singularity of type $\Sigma$ at the point $[1:0:0]$ if $t=0$ as in Definition~\ref{def:degenerations-Horikawa-we-compute-stable-replacement-of}. The divisor $\mathfrak{D}$ is Cartier and $K_\mathfrak{F}$ is $\mathbb{Q}$-Cartier.

If $\mathfrak{C}$ is the scheme associated to the ideal $(t^d,y^q,z^p)$, then define $\mathfrak{F}':=\mathrm{Bl}_\mathfrak{C}\mathfrak{F}$, and let $\mathfrak{D}'\subseteq\mathfrak{F}'$ be the strict transform of $\mathfrak{D}$. Then the fiber of $\left(\mathfrak{F}',\frac{1}{2}\mathfrak{D}'\right)\rightarrow\mathbb{P}(U_\Sigma)_{\mathrm{reg}}\times\Delta$ over $(u,0)\in\mathbb{P}(U_\Sigma)_{\mathrm{reg}}\times\Delta$ is the gluing of $\left(Y,G+\frac{1}{2}\mathcal{B}'|_Y\right)$ and $\left(Z,E+\frac{1}{2}\mathcal{B}'|_Z\right)$ as in \S\,\ref{subsec:computation-stable-replacement}. We have that $K_{\mathfrak{F}'}$ and $\mathfrak{D}'$ are both $\mathbb{Q}$-Cartier.

In particular, for $N$ large enough and divisible by $p$ and $q$ across the eight singularity types, we have that $N\left(K_{\mathfrak{F}'}+\frac{1}{2}\mathfrak{D}'\right)$ is Cartier and it restricts to the fibers $F\subseteq\mathfrak{F}'$ giving the Cartier divisor $N\left(K_F+\frac{1}{2}D\right)$. Then $\left(\mathfrak{F}',\frac{1}{2}\mathfrak{D}'\right)\rightarrow\mathbb{P}(U)_{\mathrm{reg}}\times\Delta$ is a family of KSBA stable pairs as in \S\,\ref{subsec:computation-stable-replacement}. Both $\mathfrak{F}'\rightarrow\mathbb{P}(U)_{\mathrm{reg}}\times\Delta$ and $\mathfrak{D}'\rightarrow\mathbb{P}(U)_{\mathrm{reg}}\times\Delta$ are flat as they are dominant morphisms from integral schemes to normal schemes with reduced fibres of constant dimension \cite[Lemma~10.12]{HKT09}. In particular, $\left(\mathfrak{F}',\frac{1}{2}\mathfrak{D}'\right)\rightarrow\mathbb{P}(U)_{\mathrm{reg}}\times\Delta$ is a well defined family of KSBA stable pairs for the Viehweg's moduli stack with $N$ as above (see Definition~\ref{def:Viehweg-moduli-stack}). In particular, it induces a morphism $f_\Sigma\colon\mathbb{P}(U)_{\mathrm{reg}}\times\Delta\rightarrow\overline{\mathbf{M}}$ to the KSBA compactification of the moduli space of Horikawa surfaces. Then, we define the boundary stratum $\mathbf{D}_\Sigma\subseteq\overline{\mathbf{M}}$ as the Zariski closure with the reduced scheme structure of the image $f_\Sigma(\mathbb{P}(U)_{\mathrm{reg}}\times\{0\})$.
\end{definition}

For the rest of this section, the goal is to prove the following result.

\begin{theorem}
\label{thm:KSBA-strata-are-divisors}
For each singularity type $\Sigma$, $\mathbf{D}_\Sigma$ is a divisor in $\overline{\mathbf{M}}$.
\end{theorem}

Let $\Sigma$ be one of the eight singularity types and denote by $\mu_\Sigma$ its Milnor number. The proof of Theorem~\ref{thm:KSBA-strata-are-divisors}, which we are about to discuss, boils down to checking the following for each singularity $\Sigma$:

\begin{itemize}

\item The dimension of the space of $\widetilde{Y}$ will be shown to be $\mu_\Sigma-2$;

\item The dimension of the space of $\widetilde{Z}$ is $29-\mu_\Sigma$, where $29$ is the rank of $h^{1,1}$ of a smooth Horikawa surface;

\item The deformations of $\widetilde{Y}$ and $\widetilde{Z}$ are independent.

\end{itemize}

We first need some preliminaries.

\begin{lemma}[{\cite[Proposition~5.1]{Has00}}]
\label{lem:dim-aut-groups-of-P(1,p,q)}
Let $a$ and $b$ be positive integers such that $1<a<b$, $\gcd(a,b)=1$. Then
\[
\dim(\mathrm{Aut}(\mathbb{P}(1,a,b)))=4+\lfloor b/a\rfloor,~\dim(\mathrm{Aut}(\mathbb{P}(1,1,a)))=5+a.
\]
In particular, $\dim(\mathrm{Aut} ( \mathbb{P}(1,1,2)))=7$ and the dimensions of $\mathrm{Aut}(\mathbb{P}(1,p,q))$ across the eight singularity types are given by the table below.
\begin{center}
\renewcommand{\arraystretch}{1.4}
\begin{tabular}{|c|c|c|c|c|c|c|c|c|}
\hline
Sing. & $E_{12}$ & $E_{13}$ & $E_{14}$ & $Z_{11}$ & $Z_{12}$ & $Z_{13}$ & $W_{12}$ & $W_{13}$ \\
\hline
$\dim(\mathrm{Aut}(\mathbb{P}(1,p,q)))$ & $6$ & $6$ & $6$ & $5$ & $5$ & $5$ & $5$ & $5$ \\
\hline
\end{tabular}
\end{center}
\end{lemma}

\begin{lemma}
\label{lem:dim-count-aut-group-Z-part}
Let $\Gamma_\Sigma$ be the subgroup of the automorphism group of $\mathbb{P}(1,1,2)$ that preserves $\mathbb{P}(U_{\Sigma,+}\oplus U_{\Sigma,0})_{\mathrm{reg}}$. Then the dimension of $\Gamma_\Sigma$ is as follows:
\begin{center}
\renewcommand{\arraystretch}{1.4}
\begin{tabular}{|c|c|c|c|c|c|c|c|c|}
\hline
Sing. & $E_{12}$ & $E_{13}$ & $E_{14}$ & $Z_{11}$ & $Z_{12}$ & $Z_{13}$ & $W_{12}$ & $W_{13}$ \\
\hline
$\dim(\Gamma_\Sigma)$ & $2$ & $2$ & $2$ & $3$ & $3$ & $3$ & $3$ & $3$ \\
\hline
\end{tabular}
\end{center}
\end{lemma}

\begin{proof}
In the current proof, we denote $d$ in Table~\ref{tbl:weights-and-degrees-for-eight-singularities} by $\deg$ instead. Let us start by describing the automorphisms in $\Gamma_\Sigma$. A generic automorphisms of $\mathbb{P}(1,1,2)$ has the following form:
\[
\varphi\colon[x:y:z]\mapsto[a'x+b'y:c'x+d'y:e'z+f'x^2+g'xy+h'y^2],
\]
where $a',b',c',d',e',f',g',h'$ are generic constants. Since $\varphi$ is injective, we must have that $e'\neq0$. Moreover $a'\neq0$ because we consider automorphisms sending $[1:0:0]$ to itself. So, after rescaling we can normalize $e'$ to $1$ obtaining
\[
\varphi\colon[x:y:z]\mapsto[\tilde{x}:\tilde{y}:\tilde{z}]=[ax+by:cx+dy:z+ex^2+fxy+gy^2]
\]
with $a\neq0$. To preserve $\mathbb{P}(U_{\Sigma,+}\oplus U_{\Sigma,0})_{\mathrm{reg}}$, first we must have that for each monomial $x^iy^jz^k$ such that $i+j+2k=10$ and $\mathrm{wt}_\Sigma(x^iy^jz^k)\geq0$, the monomials appearing in
\[
\tilde{x}^i\tilde{y}^j\tilde{z}^k=(ax+by)^i(cx+dy)^j(z+ex^2+fxy+gy^2)^k=\sum_{\ell,m,n} c_{\ell mn}x^\ell y^mz^n
\]
also have non-negative weight. With this we can explicitly describe $\varphi$: for each $x^iy^jz^k$, let $C_{ijk}$ be the set of coefficients $c_{\ell mn}$ such that $\mathrm{wt}_\Sigma(x^\ell y^mz^n)<0$. Define $I_\Sigma$ to be the ideal generated by the sets $C_{ijk}$ for all $i,j,k$. The construction of $I_\Sigma$ and a primary decomposition for it can be automatized with a computer using the following SageMath code \cite{Sag22}: (Here we use $p=3,q=4,\deg=15$ as an example, which correspond to $\Sigma=Z_{11}$. These values can be changed according to $\Sigma$.) 
\begin{verbatim}
R.<a,b,c,d,e,f,g>=PolynomialRing(QQ)
F=R.fraction_field()
S.<x,y,z>=PolynomialRing(F)
p=3; q=4; deg=15;

Coeff=[]          
for i in range(0,10+1):
    for j in range(0,10+1):
        for k in range(0,5+1):
            if i+j+2*k==10 and 0*i+p*j+q*k>=deg:
                M=x^i*y^j*z^k
                NewM=M.substitute(x=a*x+b*y,y=c*x+d*y,z=z+e*x^2+f*x*y+g*y^2);
                for v in NewM.exponents():
                    if v[0]*0+v[1]*p+v[2]*q<deg:
                        Cijk=NewM.coefficient(x^v[0]*y^v[1]*z^v[2])
                        Coeff.append(Cijk)

I=Ideal(Coeff)
P.<a,b,c,d,e,f,g>=PolynomialRing(QQ)
J=I.change_ring(P)
print(((J.radical())))
\end{verbatim}
If $J_\Sigma$ denotes the radical of $I_\Sigma$, we obtain that
\[
J_{E_{12}}=J_{E_{13}}=J_{E_{14}}=(c,e,f,ag),~J_{Z_{11}}=J_{W_{12}}=(c,e,af),~J_{Z_{12}}=J_{Z_{13}}=J_{W_{13}}=(c,e,f),
\]
where recall $a\neq0$. On the other hand, to be an element of $\Gamma_\Sigma$ we must also have that the weight zero monomials of the transformed polynomial have equal coefficients. This imposes one condition on the coefficients of $\varphi$. To understand this, it is enough to prove this for the associated Lie algebra. More precisely, the action of $\Gamma_\Sigma$ on $\mathbb{P}(U_{\Sigma,+}\oplus U_{\Sigma,0})_{\mathrm{reg}}$ gives a representation of the corresponding Lie algebra $\gamma_\Sigma$ on the tangent space $T_u(\mathbb{P}(U_{\Sigma,+}\oplus U_{\Sigma,0})_{\mathrm{reg}})$ for any $u\in\mathbb{P}(U_{\Sigma,+}\oplus U_{\Sigma,0})_{\mathrm{reg}}$. So, the Lie algebra $\gamma_\Sigma$ acts linearly on $T_u(\mathbb{P}(U_{\Sigma,+}\oplus U_{\Sigma,0})_{\mathrm{reg}})$, showing that we only need to impose one linear condition of the fact that these two coefficients are equal.

These considerations together give the dimension count for $\Gamma_{\Sigma}$ in the statement.
\end{proof}

\begin{proof}[Proof of Theorem~\ref{thm:KSBA-strata-are-divisors}]
As discussed in \S\,\ref{subsec:computation-stable-replacement}, it will be equivalent to count the dimension of the space of isomorphism classes of stable pairs $\left(X_0',\frac{1}{2}B_0'\right)$, which recall is the gluing of two pairs: $\left(Y,G+\frac{1}{2}\mathcal{B}'|_Y\right)$ and $\left(Z,E+\frac{1}{2}\mathcal{B}'|_Z\right)$.

The dimension of the space of pairs $\left(Y,G+\frac{1}{2}\mathcal{B}'|_Y\right)$ is equal to the dimension of the projectivized vector space $V_{p,q,d}$ of degree $d$ curves in $\mathbb{P}(1,p,q)$ with equal nonzero coefficient for $\theta(m_1)$ and $\theta(m_2)$ (see Lemma~\ref{lem:tail-Y-describes-all-deg-d-curves-in-P(1,p,q)}) minus the dimension of the subgroup $G_\Sigma\leq\mathrm{Aut}(\mathbb{P}(1,p,q))$ which preserves the equality of these two coefficients, and hence has codimension $1$ in $\mathrm{Aut}(\mathbb{P}(1,p,q))$ (for the dimension of the latter see Lemma~\ref{lem:dim-aut-groups-of-P(1,p,q)}). Therefore, we obtain
\begin{center}
\renewcommand{\arraystretch}{1.4}
\begin{tabular}{|c|c|c|c|c|c|c|c|c|}
\hline
Sing. & $E_{12}$ & $E_{13}$ & $E_{14}$ & $Z_{11}$ & $Z_{12}$ & $Z_{13}$ & $W_{12}$ & $W_{13}$ \\
\hline
$\dim(\mathbb{P}(V_{p,q,d}))-\dim G_\Sigma$ & $10$ & $11$ & $12$ & $9$ & $10$ & $11$ & $10$ & $11$ \\
\hline
\end{tabular}
\end{center}

The dimension of the space of pairs $\left(Z,E+\frac{1}{2}\mathcal{B}'|_Z\right)$ is equal to the dimension of the projectivized vector space of coefficients $U_{\Sigma,+}\oplus U_{\Sigma,0}$ (for this dimension we refer to the tables in Proposition~\ref{prop:classification-signs-of-deg-10-monomials}) minus the dimension of the group $\Gamma_\Sigma$ which was computed in Lemma~\ref{lem:dim-count-aut-group-Z-part}. Therefore, we obtain
\begin{center}
\renewcommand{\arraystretch}{1.4}
\begin{tabular}{|c|c|c|c|c|c|c|c|c|}
\hline
Sing. & $E_{12}$ & $E_{13}$ & $E_{14}$ & $Z_{11}$ & $Z_{12}$ & $Z_{13}$ & $W_{12}$ & $W_{13}$ \\
\hline
$\dim\mathbb{P}(U_{\Sigma,+}\oplus U_{\Sigma,0})-\dim(\Gamma_\Sigma)$ & $17$ & $16$ & $15$ & $18$ & $17$ & $16$ & $17$ & $16$ \\
\hline
\end{tabular}
\end{center}

Finally, let us discuss the gluing of the two pairs along $G$ and $E$. Let $g_p,g_q\in G$ (resp. $e_p,e_q\in E$) be the torus fixed points with singularities of type $\frac{1}{p}(1,q),\frac{1}{q}(1,p)$ (resp. $\frac{1}{p}(1,-q),\frac{1}{q}(1,-p)$). Denote by $g_b\in G$ (resp. $e_b\in E$) the point in $\mathcal{B}'|_Y\cap G$ (resp. $\mathcal{B}'|_Z\cap E$) different from $g_p,g_q$ (resp. $e_p,e_q$) (see Lemma~\ref{lem:intersection-curve-tain-with-conductor}). Then, the pointed curves $(G;g_p,g_q,g_b)$ and $(E;e_p,e_q,e_b)$ are identified via the unique isomorphism such that $g_p\mapsto e_p$, $g_q\mapsto e_q$, $g_b\mapsto e_b$. Having that the coefficients of the monomials $m_1,m_2$ are equal, implies that we fix the points $g_b$ and $e_b$. In particular, there is no moduli associated with the gluing.

Adding up the two contributions for each $\Sigma$, we obtain $27$ as the dimension of the KSBA boundary stratum $\mathbf{D}_\Sigma$.
\end{proof}


\section{Relation with the GIT compactification}
\label{sec:GIT-study}


\subsection{GIT stability of the eight singularity types}

Recall from \S\,\ref{sec:Horikawasurf} that the parameter space of Horikawa surfaces is the $32$-dimensional vector space
\[
W = 
H^0(\mathbb{P}^1,\mathcal{O}_{\mathbb{P}^1}(4))\oplus H^0(\mathbb{P}^1,\mathcal{O}_{\mathbb{P}^1}(6))\oplus H^0(\mathbb{P}^1,\mathcal{O}_{\mathbb{P}^1}(8))\oplus H^0(\mathbb{P}^1,\mathcal{O}_{\mathbb{P}^1}(10)).
\]
We have a natural action $\mathrm{GL}_2\curvearrowright W$ which is induced by linear change of coordinates in $x$ and $y$. It turns out that the isomorphism classes of Horikawa surfaces coincide with the orbits of this $\mathrm{GL}_2$-action \cite[Lemma~7]{Wen21}. Therefore, we obtain the quotient described in Definition~\ref{def:GIT}:
\begin{align*}
\overline{\mathbf{M}}^{\mathrm{git}} = 
\mathbb{P}
\left( 
\bigoplus_{k=2}^{5} 
H^0(\mathbb{P}^1,\mathcal{O}_{\mathbb{P}^1}(2k))
\right)
\big/\!\! \big/_{\mathcal{O}(1)}\mathrm{SL}_2.
\end{align*}

\begin{theorem}
\label{thm:log-canonical-or-eight-types-implies-GIT-stable}
Let $X$ be a double cover of $\mathbb{P}(1,1,2)$ with branch curve of degree $10$. If $X$ has isolated log canonical singularities or isolated singularities of type
\[
E_{12},~E_{13},~E_{14},~Z_{11},~Z_{12},~Z_{13},~W_{12},~W_{13},
\]
then $X$ is GIT stable.
\end{theorem}

\begin{proof}
First, we characterize the GIT non-stable points. It will be convenient to denote a point in $\mathbb{P}(W)$ by $[\mathbf{q}(x,y)]$, where
$$
\mathbf{q}(x,y) =
\left( q_4(x,y), q_6(x,y), q_8(x,y),q_{10}(x,y) \right)
$$
is identified with the vector in $W$ of coefficients of $q_4(x,y),\ldots,q_{10}(x,y)$. The point $[\mathbf{q}(x,y)]$ is stable if and only if, for every one-parameter subgroup $\lambda(t)$ of $\mathrm{SL}_2$, the limit
\[
\lim_{t \to 0}\lambda(t)\cdot \mathbf{q}(x,y)
\]
does not exist (see \cite[Theorem~7.4]{Muk03}). Therefore, a point is not stable if and only if there exists a one-parameter subgroup $\lambda(t)$ such that each of the following limits exist:
\begin{align*}
\lim_{t \to 0}\lambda(t)\cdot q_4(x,y), &&
\lim_{t \to 0}\lambda(t)\cdot q_6(x,y), && 
\lim_{t \to 0}\lambda(t)\cdot q_8(x,y), &&
\lim_{t \to 0}\lambda(t)\cdot q_{10}(x,y).
\end{align*}
 Up to a change of coordinates, we can suppose that 
$\lambda(t) = \text{diag}(t^a, t^{-a})$ with $a >0$.  The existence of the limits implies that $q_4(x,y),\ldots,q_{10}(x,y)$ can be written as follows:
\begin{align*}
q_4(x,y)    = x^2h_2(x,y), &&
q_6(x,y)    = x^3h_3(x,y), &&
q_8(x,y)    = x^4h_4(x,y), &&
q_{10}(x,y) = x^5h_{5}(x,y),
\end{align*}
where the polynomials $h_i(x,y)$ are homogeneous of degree $i$. Suppose that $h_i(x,y)$ are generic and define
$$
[\mathbf{h}(x,y)]:=
[ x^2h_2(x,y): x^3h_3(x,y): x^4h_4(x,y):x^5h_{5}(x,y)].
$$
Then, a point $[\mathbf{q}(x,y)] \in \mathbb{P}(W)$ is not stable if and only if it is in the closure of the $\mathrm{SL}_2$-orbit of $[\mathbf{h}(x,y)]$, because the non-stable locus is closed in $\mathbb{P}(W)$. 

Let $X_{\mathbf{h}}$ be the hypersurface in $\mathbb{P}(1,1,2,5)$ given by
$$
w^2 = z^5 + x^2h_2(x,y)z^3 + x^3h_3(x,y)z^2 + x^4h_4(x,y)z+ x^5h_{5}(x,y).
$$
The above surface has a singularity at $[0:1:0:0]$. So, if we consider the affine patch associated to $y\neq0$, then the affine equation of our singularity 
\[
w^2 = z^5 + x^2h_2(x,1)z^3 + x^3h_3(x,1)z^2 + x^4h_4(x,1)z+ x^5h_{5}(x,1),
\]
can be written as
\begin{align*}
w^2 = p_5(x,z) + r(x,z), && \deg{(r(x,z))} > 5,
\end{align*}
where $p_5(x,z)$ is a homogeneous polynomial of degree $5$.
Therefore, by definition, we have that $X_{\mathbf{h}}$ has either a $N_{16}$ singularity \cite[Page~13]{Arn76} or a degeneration of it. By using Arnold's work in \cite{Arn76}, we know that the Milnor number and the modality of the $N_{16}$ singularity are $16$ and $3$ respectively.   As the Milnor number $\mu$ and the modality $m$ are upper semicontinuous invariants, see \cite[\S I.2.1]{GLS07},  any non-stable surface with isolated singularities must have a singularity that satisfies $\mu\geq16$ and $m\geq3$. 
On the other hand, the classification in \cite{Arn76} and \cite{LR12} imply that the isolated log canonical singularities and the eight singularity types in the statement satisfy the inequalities $\mu \leq 14$ and $m\leq 1$.  Therefore, they are stable.
\end{proof}


\subsection{Extending the morphism from KSBA to GIT}
\label{sec:ExtendKSBAGIT}

\begin{theorem}
\label{thm:partial-extension-from-KSBA-to-GIT}
The rational map $\overline{\mathbf{M}}\dashrightarrow\overline{\mathbf{M}}^{\mathrm{git}}$ extends to a dense open subset of each of the eight boundary divisors $\mathbf{D}_{\Sigma}\subseteq\overline{\mathbf{M}}$ in Definition~\ref{def:the-eight-boundary-divisors}.
\end{theorem}

To prove this we need a preliminary lemma, which is a slight generalization of \cite[Lemma~3.18]{AET23} (see also \cite[Theorem~7.3]{GG14}).

\begin{lemma}
\label{lemma:extensionDVR}
Let $X$ and $Y$ be proper varieties with $X$ normal. Let $\varphi\colon X\dashrightarrow Y$ be a rational map which is regular on an open dense subset $U\subseteq X$. Let $(C,0)$ be a regular curve and $f\colon C\rightarrow X$ a morphism whose image meets $U$. Let $g_f\colon C\rightarrow Y$ be the unique extension of $\varphi\circ f$, which exists by the properness of $Y$.

Let $V\subseteq X$ be another dense open subset containing $U$. Assume that for all $f$ with the same $f(0)\in V$, there are only finitely many possibilities for $g_f(0)$. Then $\varphi$ can be extended uniquely to a regular morphism $V\rightarrow Y$.
\end{lemma}

\begin{proof}
Following the proof of \cite[Lemma~3.18]{AET23}, let $Z\subseteq X\times Y$ be the closure of the graph of $U\rightarrow Y$. 
By hypothesis the proper birational morphism $Z\rightarrow X$ is finite on $V$, so the base change $Z_V:=Z\times_XV\rightarrow V$ is also proper, finite, and birational. As $V$ is normal, $Z_{V}\rightarrow V$ is an isomorphism by the Zariski Main Theorem, hence we obtain the claimed extension $V\rightarrow Y$ by composing $V\rightarrow Z_V$ with the restriction to $Z_V$ of $Z\rightarrow Y$.
\end{proof}

\begin{proof}[Proof of Theorem~\ref{thm:partial-extension-from-KSBA-to-GIT}]
Fix one of the eight singularity types $\Sigma$. 
A dense open subset $\mathbf{D}_\Sigma^\circ\subseteq\mathbf{D}_\Sigma$ parametrizes stable pairs $\left(X_0',\frac{1}{2}B_0'\right)$ given by the gluing of
\[
\left(Y,G+\frac{1}{2}\mathcal{B}'|_Y\right)~\textrm{and}~\left(Z,E+\frac{1}{2}\mathcal{B}'|_Z\right)
\]
as discussed in \S\,\ref{subsec:computation-stable-replacement}. We want to show that the birational map 
$\overline{\mathbf{M}}\dashrightarrow
\overline{\mathbf{M}}^{\mathrm{git}}
$, which is an isomorphism on $\mathbf{M}$, extends to a birational morphism on $\mathbf{M}\amalg\coprod_\Sigma\mathbf{D}_\Sigma^\circ$, which is open in $\overline{\mathbf{M}}$.

To prove this, let $x\in\mathbf{D}_\Sigma^\circ$ and consider an arbitrary $f\colon (C,0)\rightarrow\overline{\mathbf{M}}$, where $C$ is a regular curve, $0\in C$, $f(0)=x$, and $f(C)\cap\mathbf{M}\neq\emptyset$. Let $g_f\colon C\rightarrow\overline{\mathbf{M}}^{\mathrm{git}}$ be the unique extension. Then there is only one possibility for $g_f(0)$, which parametrizes the following GIT stable orbit. The point $x$ parametrizes a pair $\left(X_0',\frac{1}{2}B_0'\right)$ given by the gluing of $\left(Y,G+\frac{1}{2}\mathcal{B}'|_Y\right)$ and $\left(Z,E+\frac{1}{2}\mathcal{B}'|_Z\right)$. Recall that $Z$ is the weighted blow up of $\mathbb{P}(1,1,2)$ at the point $\xi=[1:0:0]$, and under this blow up the curve $\mathcal{B}'|_Z$ is mapped to a curve in $\mathbb{P}(1,1,2)$ with equation given by
\[
(\pi_+(u)+\pi_0(u))(x,y,z)=0,
\]
where $u\in\mathbb{P}(U_{\Sigma})_{\mathrm{reg}}$ (see \S\,\ref{subsec:def-of-families}). This has a unique singularity of type $\Sigma$ at $[1:0:0]$, which we know is GIT stable by Theorem~\ref{thm:log-canonical-or-eight-types-implies-GIT-stable}. In other words, $g_f(0)$ can be uniquely reconstructed from $\left(Z,E+\frac{1}{2}\mathcal{B}'|_Z\right)$, which only depends of the point $x$ an not from the choice of $f\colon(C,0)\rightarrow(\overline{\mathbf{M}},x)$. Since $\overline{\mathbf{M}}$
is normal (see Definition~\ref{def:KSBA-comp-mod-Horikawa-surf}), we are done by Lemma~\ref{lemma:extensionDVR}.
\end{proof}


\section{Limit mixed Hodge structure of the degenerations}
\label{sec:Hodge-theory-part}

Next, we study the behavior of the Hodge structure associated with our stable surfaces. Let $f\colon\mathcal{X}\to\Delta$ be a semistable degeneration with central fiber $X_0=f^{-1}(0)$. Let $X_{\eta}=f^{-1}(\eta)$ be a generic fiber of $f$ and 
$H^k_{\lim}(X_{\eta},\mathbb Q)$ denote the $\mathbb Q$-limit mixed Hodge structure of  $R^kf_{*}(\mathbb Q)$, i.e.~the underlying $\mathbb Q$-vector space is $H^k(X_{\eta},\mathbb Q)$, but the Hodge and weight filtrations arise
from the asymptotic behavior of the period map. See \cite{Mor84,PS08} for an introduction.

\begin{theorem}
\label{thm:LMHS-is-pure}
Let $\pi\colon\mathcal{S}\to\Delta$ be a one-parameter
degeneration of complex projective surfaces which is smooth over $\Delta^* =\Delta\setminus\{0\}$ such that
\begin{itemize}
    \item[(a)] If $t\neq 0$ then $S_t=\pi^{-1}(t)$ has geometric genus $2$. 
    \item[(b)] The central fiber $S_0=\pi^{-1}(0)$ is the union of two irreducible components
    $\widetilde{Y}$ and $\widetilde{Z}$, each of which has $h^2(\mathcal{O})=1$ and at worst
    rational singularities.
\end{itemize}
Then, the local system $\mathcal{V}_{\mathbb Q} = R^2\pi_*(\mathbb Q)$ over $\Delta^*$ has finite monodromy.
\end{theorem} 
\begin{proof}
Let us consider a semistable degeneration
\begin{equation}
\label{eq:hodge-ssd}
\begin{tikzcd}
\widehat{\mathcal{S}} \arrow{r} \arrow{d} & \mathcal{S} \arrow{d} \\
\widetilde{\Delta} \arrow{r}& \Delta
\end{tikzcd}
\end{equation}
where 
$\widetilde{\Delta}\rightarrow\Delta$ is a morphism
of the form $t\mapsto t^n$ for some $n\geq 1$ and the central fiber $\widehat{S}_0$ is reduced and simple normal crossing. In particular, in order to prove that $\mathcal{V}_\mathbb{Q}$ has finite local monodromy it is sufficient to
prove that the corresponding local system attached to $\widehat{\mathcal{S}}\to\widetilde{\Delta}$ has trivial local monodromy operator $T$.

\par Let
\begin{equation*}
\widehat{S}_0=\bigcup_{i=1}^n\widehat{S}_{0i}
\end{equation*}
be the decomposition into irreducible components of the central fiber of $\widehat{\mathcal{S}}\rightarrow\widetilde{\Delta}$. Given $\eta\in\widetilde{\Delta}\setminus\{0\}$, by \cite[Page~118]{Mor84} we have that
\begin{equation}
\label{eq:Morrison-inequality-geometric-genera}
p_g(\widehat{S}_{\eta})\geq\sum_{i=1}^np_g(\widehat{S}_{0i}),
\end{equation}
and equality holds if and only if $N=\log(T)=0$. So let us prove that equality holds.

By the semistable reduction process, we have that the surfaces $\widetilde{Y}$ and $\widetilde{Z}$ are birational to $\widehat{S}_{0j}$ and $\widehat{S}_{0k}$ for some distinct $j,k\in\{1,\ldots,n\}$. Since $\widetilde Y$ and $\widetilde Z$ have rational singularities, we can conclude that $p_g(\widehat{S}_{0j})=h^2(\mathcal{O}_{\widetilde{Y}})$ and $p_g(\widehat{S}_{0k})=h^2(\mathcal{O}_{\widetilde{Z}})$, which are both equal to $1$ by hypothesis.
Thus,
\[
2=p_g(\widehat{S}_{\eta})\geq\sum_{i=1}^np_g(\widehat{S}_{0i})\geq h^2(\mathcal{O}_{\widetilde{Y}})+h^2(\mathcal{O}_{\widetilde{Z}})=2,
\]
and hence equality holds in \eqref{eq:Morrison-inequality-geometric-genera}.
\end{proof}

\par In particular, this theorem applies to the one-parameter stable degeneration of Horikawa surfaces whose central fiber $S_0$ is in the form $\widetilde{Y}\cup\widetilde{Z}$ as  described in \S\,\ref{subsec:computation-stable-replacement}.  In this case, 
\begin{itemize}
    \item The generic surface is a smooth Horikawa surface which has $p_g=2$;
    \item The surface $\widetilde{Y}$ is an ADE K3 surface by Proposition~\ref{prop:double-cover-of-Z-invariants};
    \item The surface $\widetilde{Z}$ has only finite cyclic quotient singularities by Proposition~\ref{prop:Ytilde-and-Ztilde-have-quotient-singularities} (hence rational singularities by \cite[Proposition~5.15]{KM98}) and
    $h^2(\mathcal{O}_{\widetilde{Z}})=1$ by Proposition~\ref{prop:double-cover-of-Z-invariants}.
\end{itemize}
Looking ahead to Theorem~\ref{thm:hodge-main-2}, we note that
the mixed Hodge structures on $H^2(\widetilde Y,\mathbb Q)$ and $H^2(\widetilde Z,\mathbb Q)$ are pure of weight 2.   This
is a well-known result in the case of ADE K3 surfaces. 
On the other hand, since $\widetilde Z$ has only finite cyclic
quotient singularities, it is a K\"ahler V-manifold, and hence
$H^2(\widetilde Z,\mathbb Q)$ admits a pure Hodge structure of weight $2$.

\medskip

\par To continue, we recall the following result of Griffiths.

\begin{theorem}[{\cite[\S\,4.11]{Sch73}}]
\label{thm:Sch73}
Let $\varphi\colon\Delta^*\to\Gamma\backslash\mathcal{D}$ be the period map of a variation
of pure Hodge structure over the punctured disk $\Delta^*=\Delta\setminus\{0\}$.
If the local monodromy operator $T$ of $\varphi$ has finite
order then $\varphi$ extends holomorphically to the disk $\Delta$.
\end{theorem}

\begin{corollary}
\label{lemma:LMHS-ext}
The period maps defined by the families
$\mathcal{S}$ and $\widehat{\mathcal{S}}$ of Theorem~\ref{thm:LMHS-is-pure}
extend holomorphically to the full disk.  The limit mixed Hodge structure of the semistable degeneration $\widehat{\mathcal{S}}$ is pure.
\end{corollary}

\begin{proof}
The proof of Theorem~\ref{thm:LMHS-is-pure} shows that both 
$\mathcal{S}$ and $\widehat{\mathcal{S}}$ have finite local monodromy. Hence, apply Theorem~\ref{thm:Sch73}.  In the case of $\widehat{\mathcal{S}}$, the local monodromy $T=e^N=\mathrm{id}$ and hence the limit mixed Hodge structure is pure by Theorem 6.16 of ~\cite{Sch73}.
\end{proof}

\par Returning to the first paragraph of this section, let 
$f\colon\mathcal{X}\to\Delta$ be a semistable degeneration with central fiber
$X_0$ and $X_{\eta}$ be a generic fiber of $\mathcal{X}$.  Let
$
 H^k(X_0,\mathbb Q)\to
 H^k(X_{\eta},\mathbb Q)
$
denote the composite map
$$
H^k(X_0,\mathbb Q)\stackrel{\cong}{\longrightarrow}
H^k(\mathcal{X},\mathbb Q)\to 
 H^k(X_{\eta},\mathbb Q)
$$ 
defined via the inclusion $X_{\eta}\hookrightarrow\mathcal{X}$ and the
retraction $\mathcal{X}\to X_0$.

\begin{theorem}[Clemens--Schmid Sequence, {\cite{Mor84,PS08,dCM14}}]
Let $f\colon\mathcal{X}\to\Delta$ be a
semistable degeneration and $d=\dim_{\mathbb C}(\mathcal{X})$.  
Then, 
$$
 \cdots\to H_{2d-k}(X_0,\mathbb Q)
 \to H^k(X_0,\mathbb Q)\to H^k_{\lim}(X_{\eta},\mathbb Q)
 \stackrel{N}{\to}
 H^k_{\lim}(X_{\eta},\mathbb Q)\to H_{2d-k-2}(X_0)\to\cdots
$$
is an exact sequence of mixed Hodge structures (after appropriate Tate twists), where $T=e^N$ denotes the local monodromy of $R^kf_{*}(\mathbb{Q})$.
\end{theorem}

\par We now specialize the previous theorem to the case where $d=3$, $k=2$ and $N=0$ on $H^2_{\lim}(X_{\eta},\mathbb Q)$.  Let $\mathbb Q(\ell)$ denote the pure $\mathbb Q$-Hodge structure of type $(-\ell,-\ell)$ of  rank 1 with $\mathbb Q$-structure $(2\pi i)^{\ell}\mathbb Q\subseteq\mathbb C$.  Then, 
\[
    0\to H^0_{\lim}(X_{\eta},\mathbb Q)\to
    H_4(X_0,\mathbb Q)\to
    H^2(X_0,\mathbb Q)\to
    H^2_{\lim}(X_{\eta},\mathbb Q)\to 0
\]
is an exact sequence.  The local system $R^0f_{*}(\mathbb Q)$ over $\Delta^*$
is the constant variation of Hodge structure $\mathbb Q(0)$, and hence 
the previous sequence becomes
\[
    0\to\mathbb Q(0)\to
    H_4(X_0,\mathbb Q)\to
    H^2(X_0,\mathbb Q)\to
    H^2_{\lim}(X_{\eta},\mathbb Q)\to 0.
\]
Adding the correct Tate twists
\cite[Page~108]{Mor84}, the sequence
becomes:
\begin{equation}
\label{eq:sequence-after-Tate-twists}
    0\to \mathbb Q(0)\stackrel{(-2,-2)}{\longrightarrow}
    H_4(X_0,\mathbb Q)\stackrel{(3,3)}{\longrightarrow}
    H^2(X_0,\mathbb Q)\stackrel{(0,0)}{\longrightarrow}
    H^2_{\lim}(X_{\eta},\mathbb Q)\to 0.
\end{equation}

\par To simplify the previous sequence, we note that since we are considering a degeneration of 
surfaces, it follows that (see \cite[Page~117]{Mor84})
\[
           F^{-1}\mathrm{Gr}^W_{-4}H_4(X_0,\mathbb C)=0,
\]
and hence $\mathrm{Gr}_{-4}^WH_4(X_0,\mathbb Q)$ is pure of type $(-2,-2)$, i.e.
\begin{equation}
\label{eq:iso-of-GrW-4}
\mathrm{Gr}_{-4}^WH_4(X_0,\mathbb Q)\cong\mathbb Q(2)^{\oplus(r+1)}
\end{equation}
for some integer $r\geq0$. Therefore, combining \eqref{eq:sequence-after-Tate-twists} with \eqref{eq:iso-of-GrW-4}, we obtain an exact sequence of pure Hodge structures of weight $2$
\begin{equation}
\label{eq:clemens-schmid-2}
    0\to \mathbb Q(-1)^{\oplus r}\to \mathrm{Gr}^W_2 H^2(X_0,\mathbb Q)\to 
     H^2_{\lim}(X_{\eta},\mathbb Q)\to 0.
\end{equation}

\par To show that this sequence splits, we recall the following.

\begin{theorem}[{\cite{Del71,Del74}}]
The mixed Hodge structure on the rational cohomology of a complex algebraic variety is graded-polarizable.
\end{theorem}

\par Accordingly, after selecting a choice of polarization of 
$\mathrm{Gr}^W_2 H^2(X_0,\mathbb Q)$ we obtain a direct sum
decomposition 
\begin{equation}
\label{eq:clemens-schmid-1}
    \mathrm{Gr}^W_2 H^2(X_0,\mathbb Q)\cong
    \mathbb Q(-1)^{\oplus r}\oplus (\mathbb Q(-1)^{\oplus r})^{\perp},\qquad
    (\mathbb Q(-1)^{\oplus r})^{\perp}\cong
    H^2_{\lim}(X_{\eta},\mathbb Q).
\end{equation}

\begin{definition}
\label{def:transcendental-part}
Let $A$ be a $\mathbb Q$-Hodge structure of weight 2 with $F^3 A=0$. Then, the \emph{transcendental part of $A$}, denoted by $T[A]$, is the smallest $\mathbb{Q}$-sub-Hodge structure of $A$ such that $F^2 A\subseteq T[A]_{\mathbb C}$.  
\end{definition}

\begin{lemma}
\label{lem:direct-sum} 
Suppose that $A$ and $B$ are pure 
$\mathbb Q$-Hodge structures of weight $2$ such that 
$F^3 A = F^3 B =0$. Then $T[A\oplus B] = T[A]\oplus T[B]$.  
\end{lemma}
\begin{proof} 
The Hodge filtration is an exact functor from the category of mixed Hodge structures to the category of $\mathbb C$-vector spaces. In particular, $F^2(A\oplus B) = F^2 A\oplus F^2 B$. By definition, 
$T[A]_{\mathbb C}\supseteq F^2 A$ and 
$T[B]_{\mathbb C}\supseteq F^2 B$ and hence
$(T[A]\oplus T[B])_{\mathbb C}
\supseteq F^2A\oplus F^2B = F^2(A\oplus B)$.   Therefore
$T[A]\oplus T[B]\supseteq T[A\oplus B]$.
On the other hand, $T[A\oplus B]\cap(A\oplus 0)$ is a sub-Hodge structure 
of $A\oplus 0$ containing $F^2A$. So 
$T[A \oplus B]\cap(A \oplus 0) \supseteq F^2 A \oplus 0$, which, obviously implies 
that $T[A \oplus B] \supseteq T[A]\oplus 0$. Symmetrically, 
$T[A \oplus B] \supseteq 0\oplus T[B]$, and hence 
$T[A\oplus B] \supseteq T[A]\oplus T[B]$.
\end{proof}

By applying Lemma~\ref{lem:direct-sum} to \eqref{eq:clemens-schmid-1} we obtain the following result.

\begin{corollary}
\label{cor:trans-part-GrWH2=trans-H2lim}
In the setting of Equation~\eqref{eq:clemens-schmid-1},
$T[\mathrm{Gr}^W_2 H^2(X_0,\mathbb Q)]
    \cong T[H^2_{\lim}(X_{\eta},\mathbb Q)]$.
\end{corollary}

\par As a prelude to the next result, we recall that if 
$\mathscr{S}
=\mathscr{S}_1\cup \mathscr{S}_2$ is the  union of  non-singular projective surfaces intersecting transversely then there exists a Mayer--Vietoris sequence
$$
    \ldots\to H^1(\mathscr{S}_1\cap \mathscr{S}_2)\to H^2(\mathscr{S})\to 
    H^2(\mathscr{S}_1)\oplus H^2(\mathscr{S}_2)\to H^2(\mathscr{S}_1\cap \mathscr{S}_2)\to\ldots
$$    
With the exception of $H^2(\mathscr{S})$, 
all of the terms in this sequence carry pure Hodge structures of weight equal to the cohomological degree.  Moreover, all of these maps are morphisms of mixed Hodge structure.  Therefore,
$$
    \mathrm{Gr}^W_2 H^2(\mathscr{S})\cong 
    \ker(H^2(\mathscr{S}_1\sqcup\mathscr{S}_2)\to H^2(\mathscr{S}_1\cap \mathscr{S}_2)),
$$ 
since $H^2(\mathscr{S}_1\sqcup\mathscr{S}_2)
=H^2(\mathscr{S}_1)\oplus H^2(\mathscr{S}_2)$.
Equivalently, after extending the definition of the N\'eron--Severi group additively across disjoint unions, the previous equation becomes 
$$
    \mathrm{Gr}^W_2 H^2(\mathscr{S},\mathbb Q)\cong 
    \ker(T[H^2(\mathscr{S}_1\sqcup \mathscr{S}_2)]
         \oplus{\rm NS}_{\mathbb Q}(\mathscr{S}_1\sqcup \mathscr{S}_2)\to 
         H^2(\mathscr{S}_1\cap \mathscr{S}_2,\mathbb Q)).
$$ 
In particular, since elements of $T[H^2(\mathscr{S}_i,\mathbb Q)]$
vanish upon pullback to $H^2(\mathscr{S}_1\cap \mathscr{S}_2,\mathbb Q)$
it follows that
$$
T[\mathrm{Gr}^W_2 H^2(\mathscr{S},\mathbb Q)]\cong 
T[H^2(\mathscr{S}_1\sqcup\mathscr{S}_2,\mathbb Q)].
$$
More generally, we have the following.

\begin{lemma}
\label{lem:normal-crossing-surface} 
Let $\mathscr{S}$ be a
projective surface which has only simple normal crossing singularities.  Let 
$\mathscr{S}=\cup_i\, \mathscr{S}_i$ denote the decomposition of 
$\mathscr{S}$ into irreducible components.  Then, 
$$
     T[\mathrm{Gr}^W_2 H^2(\mathscr{S},\mathbb Q)] \cong
     \bigoplus_i\, T[H^2(\mathscr{S}_i,\mathbb Q)].
$$
\end{lemma}
\begin{proof} 
Let 
$$
\Sigma_0 = \bigsqcup_i\, \mathscr{S}_i,\qquad 
\Sigma_1 =\bigsqcup_{i<j}\, \mathscr{S}_i\cap \mathscr{S}_j.
$$
Then, via the theory of semisimplical varieties (see \cite[\S\,11]{Car85}),
$$
\mathrm{Gr}^W_2 H^2(\mathscr{S},\mathbb Q) \cong \ker(H^2(\Sigma_0,\mathbb Q)\xrightarrow{\delta^*} H^2(\Sigma_1,\mathbb Q)),
$$
where the map $\delta^*$ is constructed from an alternating sum of pullbacks along the inclusion maps 
$\mathscr{S}_i\cap \mathscr{S}_j\hookrightarrow \mathscr{S}_i$.

\par In analogy with our previous discussion of the case where 
$\mathscr{S}$ had only two irreducible components, 
$$
    H^2(\Sigma_0,\mathbb Q) = T[H^2(\Sigma_0,\mathbb Q)]\oplus
                         {\rm NS}_{\mathbb Q}(\Sigma_0).
$$
Likewise, the pullback of an element of $T[H^2(\mathscr{S}_i)]$  
along the inclusion map $\mathscr{S}_i\cap \mathscr{S}_j
\hookrightarrow \mathscr{S}_i$ is zero.
Therefore,  $T[H^2(\Sigma_0,\mathbb Q)]\subseteq\ker(\delta^*)$ and hence
$$
   \ker(H^2(\Sigma_0,\mathbb{Q})\xrightarrow{\delta^*} H^2(\Sigma_1,\mathbb{Q}))
= T[H^2(\Sigma_0,\mathbb Q)]\oplus\ker({\rm NS}(\Sigma_0)_{\mathbb Q}\xrightarrow{\delta^*} H^2(\Sigma_1,\mathbb{Q})).
$$
Thus,
\[
T[\mathrm{Gr}^W_2 H^2(\mathscr{S},\mathbb{Q})]
=T[H^2(\Sigma_0,\mathbb Q)] \cong\bigoplus_i\, T[H^2(\mathscr{S}_i)].\qedhere
\]
\end{proof}

\begin{theorem}
\label{thm:hodge-main-2}  
Let
$f\colon\mathcal{X}\to\Delta$ be a semistable degeneration of projective surfaces
with trivial local monodromy (as in the paragraph above \eqref{eq:clemens-schmid-2}).  Let 
$
        X_0 = \cup_j\, D_j
$
be the decomposition of the central fiber $X_0$  into irreducible components.
Then,
$$
     T[H^2_{\lim}(X_{\eta},\mathbb Q)]
     \cong T[\mathrm{Gr}^W_2 H^2(X_0,\mathbb Q)]\cong
     \bigoplus_j\, T[H^2(D_j,\mathbb Q)].
$$
\end{theorem}
\begin{proof} 
The first isomorphism is Corollary~\ref{cor:trans-part-GrWH2=trans-H2lim}. 
The second isomorphism is Lemma~\ref{lem:normal-crossing-surface}.
\end{proof}

\begin{corollary}
\label{cor:transcendental-part-breaks-over-Q}
Let $\pi\colon\mathcal{S}\to\Delta$ be as in Theorem~\ref{thm:LMHS-is-pure} and $\widehat{\mathcal{S}}\to\widetilde{\Delta}$ be the corresponding semistable degeneration \eqref{eq:hodge-ssd}.   Then, 
$$
     T[H^2_{\lim}(\widehat{S}_\eta,\mathbb Q)]
     \cong T[H^2(\widetilde Z,\mathbb Q)]
    \oplus T[H^2(\widetilde Y,\mathbb Q)].
$$
\end{corollary}
\begin{proof} 
By Theorem~\ref{thm:hodge-main-2}, the left hand side is equal to the 
sum of the transcendental parts of the irreducible components of the central fiber $\widehat{S}_0$ of $\widehat{\mathcal{S}}\to\widetilde{\Delta}$.  If
$D$ is an irreducible component of $\widehat{S}_0$ with geometric genus zero then 
$T[H^2(D,\mathbb Q)]=0$.  By the proof of Theorem \ref{thm:LMHS-is-pure}, this is true
for every irreducible component of $\widehat{S}_0$ except for the two corresponding to 
$\widetilde Z$ and $\widetilde Y$.  Since the transcendental part of $H^2$ of a surface
is a birational invariant, the result follows.
\end{proof}

\begin{remark} In~\cite{KLS21}, the authors consider various generalizations
of the Clemens--Schmid sequence using the decomposition theorem.  Of particular
relevance to the class of degenerations considered in this paper is Corollary~9.9~(i),
which asserts the following:  Let $f\colon\mathcal{X}\to\Delta$ be a flat projective family with
$\mathcal X-X_0$ smooth.   Assume that $X_0$ is reduced with semi-log canonical singularities
and $\mathcal{X}$ is normal and $\mathbb{Q}$-Gorenstein.  Then, 
$\mathrm{Gr}^0_F H^k(X_0)\cong\mathrm{Gr}^0_F H^k_{\lim}(X_t)$ for all $k$. One 
consequence of this result is the equality of the Hodge--Deligne
numbers $h^k(X_0)^{p,q} = h^k_{\lim}(X_t)^{p,q}$ for $pq=0$.
\end{remark}

\section{Birational type of limit surfaces}
\label{sec:birat-type-limit-surf}

\par In this section, we show that the minimal model of a generic surface $S_0$ in the sense of Definition~\ref{def:sigma-generic} of type $Z_{11}$, $Z_{12}$, $Z_{13}$, $W_{12}$, $W_{13}$ is a K3 surface.  In fact, our method constructs the minimal model as the minimal resolution of a double sextic.  Our techniques do not apply to the $E_{12}$, $E_{13}$, $E_{14}$ cases because the method does not produce such plane curve.

\begin{proposition}
\label{prop:birational-equiv}
Let $\Sigma\in\{Z_{11},Z_{12},Z_{13},W_{12},W_{13}\}$ and let $S_0$ be the central fiber of the one-parameter degeneration $\mathcal{S}=\mathcal{S}(t\star u)\to\Delta$ as in Definition~\ref{def:degenerations-Horikawa-we-compute-stable-replacement-of} with $u\in(U_\Sigma)_{\mathrm{reg}}$ $\Sigma$-generic. In particular, $S_0$ is the double cover of $\mathbb{P}(1,1,2)$ with branch curve $B_0$ in \eqref{eq:limit-branch-curve}. Then, $S_0$ is birational to a K3 surface with ADE singularities which is the double cover of $\mathbb{P}^2$ branched along a plane sextic  $V(H_{\Sigma})$. Furthermore, a plane sextic $C$ is projectively equivalent to $V(H_\Sigma)$ if and only if

\begin{itemize}

\item[(i)] $\Sigma=Z_{11}$, $C$ is smooth, and there exists a line $L$ such that $L \cap C$ is the union of three points of multiplicities three, two, and one.

\item[(ii)] $\Sigma=W_{12}$, $C$ is smooth, and there exists a line $L$ such that $L \cap C$ is the union of two points of multiplicities two and four.

\item[(iii)] $\Sigma=W_{13}$, $C$ has an $A_1$ singularity at a point $p$, and there exists a line $L$ such that $L \cap C$ is the union of $p$ with multiplicity four and another double point.

\item[(iv)] $\Sigma=Z_{12}$, $C$ has an $A_1$ singularity at a point $p$, and there exists a line $L$ such that $L \cap C$ is the union of $p$ with multiplicity three, and other two points with multiplicity two and one.

\item[(v)] $\Sigma=Z_{13}$, $C$ has an $A_2$ singularity at a point $p$, and there exists a line $L$ such that $L \cap C$ is the union of $p$ with multiplicity three, and other two points with multiplicity two and one.
\end{itemize}
Moreover, let 
$\mathbb{L}(\Sigma) \subseteq  \mathbb{P}(H^0(\mathbb{P}^2,\mathcal{O}(6)))$ 
be the locus parametrizing plane sextics that are projecively equivalent to $V(H_{\Sigma})$. Then, it holds that
\begin{center}
\renewcommand{\arraystretch}{1.4}
\begin{tabular}{|c|c|c|c|c|c|}
\hline
Sing.  & $Z_{11}$ & $Z_{12}$ & $Z_{13}$ & $W_{12}$ & $W_{13}$ \\
\hline
$\dim(\mathbb{L}(\Sigma)) - \dim(\mathrm{Aut}(\mathbb{P}^2))$ & $18$ & $17$ & $16$ & $17$ & $16$ \\
\hline
\end{tabular}
\end{center}
\end{proposition}

\begin{proof}
Consider the rational map
\begin{align*}
\mathbb{P}(1,1,2,5)&\dashrightarrow \mathbb{P}(2,2,2,6) \cong \mathbb{P}(1,1,1,3)\\
[x:y:z:w]&\mapsto[xy:y^2:z:yw]=[x_0:x_1:x_2:x_3].
\end{align*}
This is birational as it restricts to an isomorphism of the smooth affine charts where $y\neq0$ and $x_1\neq0$. Notice that this induces a birational map $\mathbb{P}(1,1,2) \dashrightarrow \mathbb{P}^2$, which at the level of the monomials $x^ay^bz^c\in U_{\Sigma,\geq0}:=U_{\Sigma,+} \cup U_{\Sigma,0}$ is given by
\[
    \mu ( 
    x^ay^bz^c
    )
     =
    \mu ( 
    (xy)^a(y^2)^{(b-a)/2}z^c
    )
    =
    x_0^ax_1^{(b-a)/2}x_2^c.
\]
Indeed, a direct inspection using Proposition~\ref{prop:classification-signs-of-deg-10-monomials} reveals that $(b-a)/2+1$ is always a non-negative integer for the $Z$ and $W$ series of singularities. Let $P_{\Sigma}(x,y,z)$ be the polynomial of degree $10$ given by  a linear combination of monomials in $U_{\Sigma, \geq 0}$ with general coefficients $a_{ijk}$. The vanishing of $P_{\Sigma}(x,y,z)$ defines the curve $B_0$ in the statement of the proposition.

The transformed polynomial
\begin{align*}
H_\Sigma(x_0,x_1,x_2):=\mu(P_\Sigma(x,y,z))\cdot x_1 =
\sum_{x^iy^jz^k \in U_{\Sigma, \geq 0}} a_{ijk}\mu(x^iy^jz^k)\cdot x_1
\end{align*}
defines a plane sextic curve (this is also verified by inspection using Proposition~\ref{prop:classification-signs-of-deg-10-monomials}), and the surface $S_0$ is birational to the double cover of $\mathbb{P}^2$ branched along $V(H_{\Sigma})$. (As remarked at the beginning of \S\,\ref{sec:birat-type-limit-surf}, this construction does not give a plane sextic for the $E$ series because the monomial $x^4z^3$ has weight zero and is transformed to $\mu(x^4z^3)\cdot x_1=x_0^4x_1^{-1}x_2^3$.) To describe $S_0$, we study the singularities of this plane sextic as follows. We choose a specialization $F_\Sigma$ of the polynomial $H_\Sigma$. If the curve $V(F_\Sigma)$ is smooth, then the general $V(H_\Sigma)$ is also smooth. Or, if the particular curve $V(F_\Sigma)$ has exactly one isolated $A_n$ singularity at $p$ and the general $V(H_\Sigma)$ also has one $A_n$ singularity at $p$, then the singularity of $V(H_{\Sigma})$ is unique as well. The reason is that any other singularity of $V(H_{\Sigma})$ would degenerate to $p$ along with the $A_n$ one. Therefore, the Milnor number of the singularity at $p$ of the degeneration has to be strictly larger than $n$. This is in contradiction with the characterization of $V(F_\Sigma)$, by considering the sum of the Milnor numbers of the general curve $V(H_\Sigma)$ and the fact that the Milnor number is upper semicontinuous, see \cite[Theorem~2.6]{GLS07}.

The special polynomials $F_\Sigma$ are provided below, together with the Macaulay2 code \cite{M2} that computes their singularities (we use the packages \cite{LK,Sta}).
\begin{verbatim}
needsPackage("Resultants")
needsPackage("NumericalAlgebraicGeometry")
CC[x_0,x_1,x_2]

FZ11 = x_1*(x_0^5 + x_1^5 + x_2^5) + x_2^3*x_0^3 

FZ12 = x_1*(x_2*(x_0^4 + x_1^4 + x_2^4) + x_1*(x_0^4 + x_1^4)) 
       + x_0^2*x_2^3*(x_0 + x_2)
       
FZ13 = x_1*(x_2^2*(x_0^3 + x_1^3 + x_2^3) + x_1*x_2*(x_0^3 + x_1^3) 
       + x_1*(x_0^4 + x_1^4)) + x_0^2*x_2^3*(x_0 + x_1)
       
FW12 = x_1*(x_0^5 + x_1^5 + x_2^5) + x_2^4*x_0^2 

FW13 = x_1*(x_1^5 + x_0^4*x_2 + x_2^5) + x_0^2*x_2^4 

for i in [FZ11, FZ12, FZ13, FW12, FW13] do (
    if discriminant i !=0 then 
        print("Non-zero discriminant", discriminant i )
    else (
        Q = sub(i, {x_1=>1 });
        R = sub(i, {x_2=>1 });
        solsQ = solveSystem { diff(x_0,Q), diff(x_2,Q), Q};
        solsR = solveSystem { diff(x_0,R), diff(x_1,R), R};
        print("complex solutions", solsQ, solsR)  ); )
\end{verbatim}
The chosen special polynomials $F_\Sigma$ we listed satisfy the condition that the two monomials $m_1,m_2$ of weight zero (see Proposition~\ref{prop:classification-signs-of-deg-10-monomials}) have equal non-zero coefficient. We have that $V(H_\Sigma)$ is smooth for $\Sigma=Z_{11},W_{12}$ and it has exactly one singular point if $\Sigma=Z_{12},Z_{13},W_{13}$. More specifically, we prove that $V(H_{Z_{12}}),V(H_{Z_{13}}),V(H_{W_{13}})$ have an $A_1,A_2,A_1$ singularity at $p$ respectively. This will be done on the way as we prove the claimed geometric characterization of the plane sextics.

We now prove the geometric characterization of the plane sextic $V(H_\Sigma)$ given in the statement. In what follows, for a positive integer $d$, $p_d$ and $q_d$ denote homogeneous polynomials of degree $d$.
\begin{itemize}

\item If $\Sigma = Z_{11}$, then, after enumerating the monomials in $U_{Z_{11},\geq0}$, one can observe that
\begin{equation}
\label{eq:general-sextic-curve-for-Z11}
H_{Z_{11}}=x_1p_5(x_0,x_1,x_2)+x_0^2x_2^3p_1(x_0,x_2),
\end{equation}
where $p_5(x_0,x_1,x_2) =  c_0x_0^5 + x_1q_4(x_0,x_1,x_2)$, $p_1(x_0,x_2) = c_0x_0 + c_2x_2$ with $c_0 \neq 0$. The transformations $\mu(x^5y^5)\cdot x_1=x_0^5x_1$ and $\mu(x^3yz^3)\cdot x_1=x_0^3x_2^3$ of the two weight zero monomials have equal non-zero coefficients.

First, observe that  $H_{Z_{11}}$ in \eqref{eq:general-sextic-curve-for-Z11} satisfies the claimed condition with respect to the line $L=V(x_1)$. For the converse, suppose that $C$ and $L$ are as in part~(i). Let us write the homogeneous polynomial of degree $6$ describing $C$ as
$$
x_1q_5(x_0,x_1,x_2)+q_6(x_0,x_2).
$$
Up to projectivity, we can suppose that $L=V(x_1)$, and that the intersection points $C \cap L$ with multiplicities two and three are at $[0:0:1]$ and $[1:0:0]$ respectively. In this case, $q_6(x_0,x_2)=x_0^2x_2^3q_1(x_0,x_2)$, which shows $C$ is described by the vanishing of an equation in the form \eqref{eq:general-sextic-curve-for-Z11}.

\item If $\Sigma=W_{12}$, then using the monomials in $U_{W_{12},\geq0}$ we obtain that
\[
H_{W_{12}}
=
x_1p_5(x_0,x_1,x_2)+c_0x_0^2x^4_2,
\]
where 
$p_5(x_0,x_1,x_2)=c_0x_0^5+x_1p_4(x_0,x_1,x_2)+x_2q_4(x_0,x_1,x_2)$ with $c_0 \neq 0$. The transformations $\mu(x^5y^5)\cdot x_1=x_0^5x_1$ and $\mu(x^2z^4)\cdot x_1=x_0^2x_2^4$ of the two weight zero monomials have equal non-zero coefficients given by $c_0$. The claimed characterization for the plane sextic $C$ is analogous to the one in the $Z_{11}$ case.

\item If $\Sigma = W_{13}$, then from the monomials in $U_{W_{13},\geq0}$ we obtain that
\begin{equation}
\label{eq:W13-general-sextic}
H_{W_{13}} = 
x_1 \left( x_1p_4(x_0,x_1,x_2) + x_2q_4(x_0,x_2) \right) + c_0x_0^2x_2^4,
\end{equation}
where $q_4(x_0,x_2)=c_0x_0^4+x_2q_3(x_0,x_2)$ with $c_0\neq0$. The transformations $\mu(x^4y^4z)\cdot x_1=x_0^4x_1x_2$ and $\mu(x^2z^4)\cdot x_1=x_0^2x_2^4$ of the two weight zero monomials have equal non-zero coefficients given by $c_0$. The line $L=V(x_1)$ has the claimed intersections with the plane sextic defined by \eqref{eq:W13-general-sextic}. To conclude, we have to show the point of multiplicity $4$ is an $A_1$ singularity at $[1:0:0]$. But this is clear after restricting to the affine patch $x_0\neq0$ and considering the lowest degree part.

In the other direction, by using the action of $\mathrm{SL}_3$ we can suppose that the line $L$ is $V(x_1)$, that the intersection points $C \cap L$ are $[1:0:0]$ and $[0:0:1]$ with multiplicities four and two respectively, so that the singular point $p\in C$ is supported at $[1:0:0]$. Any degree $6$ polynomial can be written as 
\begin{align*}
x_1p_5(x_0,x_1,x_2) + p_6(x_0,x_2).    
\end{align*}
The condition that $C \cap L$ has multiplicity two and four at the mentioned points implies $p_6(x_0,x_2) = ax_0^2x_2^4$ for some nonzero constant $a$. Observe that
\begin{align*}
& x_1p_5(x_0,x_1,x_2) + ax_2^4x_0^2 
\\
&= 
x_1\left( 
x_1p_4(x_0,x_1,x_2) + x_2q_4(x_0,x_2) + bx_0^5
\right) + ax_2^4x_0^2.
\end{align*}
After restricting to the affine patch $x_0\neq0$ and considering the lower degree terms we find that
\begin{align*}
& x_1\left( 
x_1p_4(1,x_1,x_2) + x_2q_4(1,x_2) + b
\right) + ax_2^4
\\
& =
x_1^2 + x_1x_2 + bx_1 + ax_2^2 + \textrm{higher order terms.}   
\end{align*}
From this local description, we conclude that $C$ is singular if and only if $b=0$. Furthermore, the singularity is a double point with non-degenerated quadratic terms. Therefore, it is an $A_1$ singularity.

\item If $\Sigma  = Z_{12}$, then using the monomials in $U_{Z_{12},\geq0}$ we obtain that
\begin{align*}
H_{Z_{12}} =
x_1\left( x_2p_4(x_0,x_1,x_2) + x_1q_4(x_0,x_1) 
\right)
+ x_0^2x_2^3p_1(x_0,x_2),
\end{align*}
where we have $p_1(x_0,x_2)=c_0x_0+c_2x_2$ and $p_4(x_0,x_1,x_2)=c_0x_0^4+x_1p_3(x_0,x_1,x_2)+x_2q_3(x_0,x_1,x_2)$ with $c_0\neq0$. The transformations $\mu(x^4y^4z)\cdot x_1=x_0^4x_1x_2$ and $\mu(x^3yz^3)\cdot x_1=x_0^3x_2^3$ of the two weight zero monomials have equal non-zero coefficients given by $c_0$. The characterization of the plane sextic $V(H_{Z_{12}})$ is analogous to the one for $W_{13}$.

\item If $\Sigma=Z_{13}$, then using the monomials in $U_{Z_{13},\geq0}$ we obtain that
$$
H_{Z_{13}} = 
x_1 \left( 
x_2^2p_3(x_0,x_1,x_2) + x_1x_2p_3(x_0,x_1) + x_1p_4(x_0,x_1)
\right)
+ x_0^2x_2^3p_1(x_0,x_1),
$$
where $p_1(x_0,x_1)=c_0x_0+c_1x_1$ and $p_4(x_0,x_1)=c_0x_0^4+x_1q_3(x_0,x_1)$ with $c_0\neq0$.
The transformations $\mu(x^4y^6)\cdot x_1=x_0^4x_1^2$ and $\mu(x^3yz^3)\cdot x_1=x_0^3x_2^3$ of the two weight zero monomials have equal non-zero coefficients given by $c_0$.
The line $L=V(x_1)$ intersects $V(H_{Z_{13}})$ as claimed in the statement. Let us explain why we have an $A_2$ singularity at $[1:0:0]$. By restricting the curve to the affine patch $x_0\neq0$, we find that the lowest degree monomials in degree $2$ and $3$ are $x_1^2,x_1x_2^2,x_1^2x_2,x_2^3$. By \cite[Lemma~1]{BW79} we recognize this is an $A_2$ singularity.

Next, we recover the above shape of the equation from the claimed incidence relation between a plane sextic $C$ and a line $L$. 
If $C$ is a general curve of degree $6$, then it can be written as
$$
x_1p_5(x_0,x_1,x_2) + p_6(x_0,x_2).
$$
Up to projective transformation, we can assume that $L=V(x_1)$ and that the two points of tangency are $[1:0:0]$ and $[0:0:1]$, so that the equation of $C$ is
$$
x_1p_5(x_0,x_1,x_2) + x_0^2x_2^3p_1(x_0,x_2),
$$
which we can rewrite as
\begin{center}
$
x_1\left( 
x_2^2p_3(x_0,x_1,x_2) + 
ax_2x_0^4 + x_2x_1p_3(x_0,x_1) +
x_1p_4(x_0,x_1) + bx_0^5
\right)
+ x_0^2x_2^3p_1(x_0,x_2).
$
\end{center}
The hypothesis that $C$ is singular at $[1:0:0]$ implies that $b=0$. By restricting to the affine patch $x_0\neq0$ we obtain 
$$
x_1\left( 
x_2^2p_3(1,x_1,x_2) + 
ax_2 + x_2x_1p_3(1,x_1) +
x_1p_4(1,x_1) 
\right)
+ x_2^3p_1(1,x_2),
$$
whose lower degree part is a linear combination of the monomials
$\{ ax_1x_2, x_1^2\}$. If $a\neq0$, then we would have an $A_1$ singularity at $[1:0:0]$, which we know it is not. Hence, we obtain $a=0$, and we can recover the equation 
$$
x_1\left( 
x_2^2p_3(x_0,x_1,x_2) + 
x_2x_1p_3(x_0,x_1) +
x_1p_4(x_0,x_1) 
\right)
+ x_0^2x_2^3p_1(x_0,x_2),
$$
which has an $A_2$ singularity at $[1:0:0]$ again by \cite[Lemma~1]{BW79}.
\end{itemize}
Finally, we discuss the dimension count for the plane sextics $V(H_\Sigma)$. For $\Sigma = Z_{11}$, the form of the polynomial $H_{Z_{11}}$ is given by \eqref{eq:general-sextic-curve-for-Z11} which has $23$ distinct monomials. We do not yet impose the equality of the coefficients of $x_1x_0^5$ and $x_0^3x_2^3$ because this can be achieved using the $\mathrm{SL}_3$-action, which we will account for at the end. We have that $\dim(\mathbb{L}(\Sigma))-\dim(\mathrm{Aut}(\mathbb{P}^2))$ is then obtained by considering the number of monomials of $H_{Z_{11}}$ in \eqref{eq:general-sextic-curve-for-Z11} up to projective scaling of the coefficients, the choice of the line $L$ (which is $V(x_1)$ in \eqref{eq:general-sextic-curve-for-Z11}),
the choice of the two points in $L$, and the $\mathrm{SL}_3$-action. These considerations yield
\[
\dim(\mathbb{L}(Z_{11}))-\dim(\mathrm{Aut}(\mathbb{P}^2))=(23-1+2+1+1)-8=18.
\]
Alternatively, fixed the plane sextic, the line $L$ is determined up to a finite choice, we choose two points on this line, and then subtract the dimension of the subgroup of automorphisms of $\mathbb{P}^2$ that preserve the two points.
An analogous argument gives the dimension of $\mathbb{L}(\Sigma)$ for $\Sigma=Z_{12},Z_{13},W_{12},W_{13}$.
\end{proof}

\begin{remark}
Let $M$ be a complex projective surface.  Then, by \cite[Lemma~3.1]{Shi08}, the transcendental lattice of $M$ is a birational invariant of $M$. This implies that the transcendental part $T[H^2(\widehat S_0,\mathbb Q)]$ considered in \S\,\ref{sec:Hodge-theory-part} depends only on the birational type the component surfaces $\widetilde Z$ and $\widetilde Y$. 
\end{remark}

Let $S_0$ be the central fiber of $\mathcal{S}\rightarrow\Delta$ as in the statement of Proposition~\ref{prop:birational-equiv}. We note that $S_0$ is a simply connected variety \cite[(B21)~Corollary]{Dim92}. As an application of the fact that $S_0$ is birational to a K3 surface we show that $\widetilde{Z}$ is also simply connected.

\begin{corollary}
If $\Sigma\in\{Z_{11},Z_{12},Z_{13},W_{12},W_{13}\}$, then the corresponding  $\widetilde{Z}$, as defined on \S\,\ref{subsec:double-cover-of-Z}, is simply connected 
and $h^{1,1}(\widetilde{Z}) = 32 -\mu_{\Sigma}$.
\end{corollary}
\begin{proof}
For a fixed $\Sigma$ as in our hypothesis, Proposition~\ref{prop:birational-equiv} the surface $\widetilde{Z}$ is birational to a K3 surface $P$ with canonical singularities.  Therefore, we have a smooth surface $W$ such that
$\widetilde{Z} \leftarrow W \rightarrow P$. The surfaces $P$ and $\widetilde{Z}$ have log terminal singularities ($P$ has ADE singularities and for $\widetilde{Z}$ see the discussion in the proof of Proposition~\ref{prop:Ytilde-and-Ztilde-have-quotient-singularities}), and $W$ is smooth.  Then, their fundamental groups are isomorphic by \cite[Theorem~1.1]{Tak03}.
Simply connectedness implies $H^1(\widetilde{Z})=0$, and since our surfaces have cyclic quotient singularities, Serre duality applies and $H^3(\widetilde{Z}) =0$
by \cite[Corollary 2.48]{PS08}.
By Corollary 
~\ref{cor:top-Eul-char-tildeZ}, we have
$\chi_{\mathrm{top}}(\widetilde{Z}) = 36 -\mu_{\Sigma}$ and $p_g=1$ 
by Proposition~\ref{prop:double-cover-of-Z-invariants}. Therefore, we obtain  $h^{1,1}(\widetilde{Z})=32 -\mu_{\Sigma}$.
\end{proof}



\newcommand{\etalchar}[1]{$^{#1}$}

\end{document}